\documentclass[11pt]{amsart}
\usepackage{amsmath, amscd, amsthm} 
\usepackage{amssymb, amsfonts} 
\usepackage{enumerate}
\usepackage{bbm}
\usepackage{color}
\usepackage[all]{xy} 
\usepackage[margin=1.1in]{geometry}
\usepackage{palatino}
\usepackage{mathrsfs}
\usepackage{booktabs}
\usepackage{aliascnt}
\usepackage{mathrsfs}
\usepackage[svgnames]{xcolor} 
\usepackage{hyperref}
\usepackage{mathrsfs}
\usepackage{tikz}
\usepackage{cleveref}
\usepackage[english]{babel}
\usepackage{url}

\usepackage{etoolbox}
\makeatletter
\patchcmd{\l@chapter}{1.0em}{0.8em}{}{}
\makeatother


\title{Stable motivic homotopy theory at infinity}

\author{Adrien Dubouloz}
\address{IMB UMR5584, CNRS, Universit{\'e} de Bourgogne Franche-Comt{\'e}, Dijon, France}
\email{Adrien.Dubouloz@u-bourgogne.fr}
\author{Fr{\'e}d{\'e}ric D{\'e}glise}
\address{ENS de Lyon, UMPA, UMR 5669, 46 all{\'e}e d'Italie, 69364 Lyon Cedex 07, France}
\email{frederic.deglise@ens-lyon.fr}
\author{Paul Arne {\O}stv{\ae}r}
\address{Department of Mathematics, University of Oslo, Norway}
\email{paularneostvar@gmail.com}

\date{\today}
\newtheorem{thm}{Theorem}[subsection]
\newtheorem{prop}[thm]{Proposition}
\newtheorem{lm}[thm]{Lemma}
\newtheorem{cor}[thm]{Corollary}

\theoremstyle{remark}
\newtheorem{rem}[thm]{Remark}

\newtheorem{ex}[thm]{Example}

\theoremstyle{definition}
\newtheorem{df}[thm]{Definition}
\newtheorem{num}[thm]{}

\numberwithin{equation}{thm}

\newtheorem{thm*}{Theorem}

\DeclareMathOperator{\Hom}{Hom}
\DeclareMathOperator{\Spec}{Spec}
\DeclareMathOperator{\uHom}{\underline{Hom}}

\DeclareMathOperator{\Th}{Th}

\DeclareMathOperator{\CHt}{\widetilde{CH}}
\DeclareMathOperator{\GW}{GW}
\DeclareMathOperator{\W}{W}

\DeclareMathOperator{\HM}{\underline H}

\DeclareMathOperator{\Pic}{Pic}
\newcommand{\KMW}{\mathrm K^{MW}}

\newcommand{\cech}{\check{\mathrm S}}
\newcommand{\cecho}{\check{\mathrm S}^{\mathrm{ord}}}

\newcommand{\derR}{\mathbf{R}}

\DeclareMathOperator{\colim}{colim}

\newcommand{\twist}[1]{\langle #1 \rangle}
\newcommand{\dtwist}[1]{\{ #1 \}}

\DeclareMathOperator{\uK}{\underline K} 

\DeclareMathOperator{\codim}{codim}
\DeclareMathOperator{\tdeg}{\widetilde{deg}}

\DeclareMathOperator{\SH}{SH}
\DeclareMathOperator{\DM}{DM}
\DeclareMathOperator{\MMAT}{MM^{AT}}
\DeclareMathOperator{\iSH}{\mathscr{SH}}

\DeclareMathOperator{\T}{\mathscr T} 

\DeclareMathOperator{\Sch}{Sch}
\newcommand{\Sche}{\operatorname{Sch}^{\mathrm{tf}}}
\DeclareMathOperator{\Sm}{Sm}

\newcommand{\HMW}{\mathbb H_{MW}}
\newcommand{\HMWx}[1]{\mathbb H_{MW,#1}}

\DeclareMathOperator{\Id}{Id}

\newcommand{\NN} {\mathbb N}
\newcommand{\ZZ} {\mathbb Z}
\newcommand{\QQ} {\mathbb Q}

\newcommand{\CC} {\mathbb C}
\renewcommand{\AA} {\mathbf A}
\newcommand{\PP} {\mathbf P}

\newcommand{\SL}{\mathbf{SL}}
\newcommand{\E}{\mathbb E}

\newcommand{\KGL}{\mathbf{KGL}}
\newcommand{\uKMW}{\underline{\mathbf K}^{MW}}

\newcommand{\un}{\mathbbm 1} 

\newcommand{\Dinj}{\Delta^{\mathrm{inj}}}    

\newcommand{\htp}{\Pi} 
\newcommand{\cohtp}{\mathrm H} 

\newcommand{\et}{\textrm{\'et}}
\newcommand{\cdh}{{\textrm{cdh}}}

\begin{document}

\begin{abstract}
In this paper, 
we initiate a study of motivic homotopy theory at infinity.
We use the six functor formalism to give an intrinsic definition of the stable motivic homotopy type at infinity 
of an algebraic variety.
Our main computational tools include cdh-descent for normal crossing divisors, Euler classes, Gysin maps, 
and homotopy purity.
Under $\ell$-adic realization, 
the motive at infinity recovers a formula for vanishing cycles due to Rapoport-Zink;
similar results hold for Steenbrink's limiting Hodge structures and Wildeshaus' boundary motives.
Under the topological Betti realization, 
the stable motivic homotopy type at infinity of an algebraic variety recovers the singular complex at infinity 
of the corresponding analytic space.
We coin the notion of homotopically smooth morphisms with respect to a motivic $\infty$-category and 
use it to show a generalization to virtual vector bundles of Morel-Voevodsky's purity theorem, 
which yields an escalated form of Atiyah duality with compact support.
Further we study a quadratic refinement of intersection degrees taking values in motivic cohomotopy groups.
For relative surfaces, 
we show the stable motivic homotopy type at infinity witnesses a quadratic version of Mumford's plumbing construction 
for smooth complex algebraic surfaces.
Our main results are also valid for $\ell$-adic sheaves, mixed Hodge modules, 
and more generally motivic $\infty$-categories.
\end{abstract}

\maketitle

\setcounter{tocdepth}{2}

\vspace{-0.2in}

\tableofcontents

\newpage

\section{Introduction}
\label{section:introduction}

Topology at infinity is essentially the study of topological properties that persistently occur in complements 
of compact sets.
For example, 
a space is intuitively simply connected at infinity if one can collapse loops far away from any 
small subspace. 
Euclidean space ${\bf R}^{n}$, 
$n\geq 3$, 
is the unique open contractible $n$-manifold that is simply connected at infinity.
In particular, 
the Whitehead manifold is not simply connected at infinity and therefore not homeomorphic to ${\bf R}^{3}$.
This article describes our first attempt at extending the theory of topology at infinity for open manifolds 
to smooth affine varieties.
An overriding goal is to develop a study of intrinsic motivic invariants which can distinguish between 
$\AA^{1}$-contractible varieties.
For background on motivic homotopy theory and $\AA^{1}$-contractible varieties we refer to the survey 
\cite{asokostvar}.
Our approach makes extensive use of the six-functor formalism in stable motivic homotopy theory, 
as developed in \cite{Ayoub,CD3}, 
and we assume some familiarity with this material.
\vspace{0.1in}

Let $S$ be a qcqs (quasi-compact quasi-separated) base scheme. 
Its stable motivic homotopy category $\SH(S)$ is a closed symmetric monoidal $\infty$-category, 
see, e.g., \cite{DRO,hoyoisquadratic,JardineMSS,robalo}.
To any separated $S$-scheme of finite type $f\colon X\to S$ we define $\Pi_{S}^{\infty}(X)$, 
the \emph{stable motivic homotopy type at infinity of $X$}, 
by the exact homotopy sequence
\begin{equation}
\label{equation:hes}
\Pi_{S}^{\infty}(X) \rightarrow f_!f^!(\un_S) \xrightarrow{\alpha_X} f_*f^!(\un_S)
\end{equation}
Here $\un_S$ is the motivic sphere spectrum over $S$, 
$f_!f^!(\un_S)$ is the stable homotopy type of $X$, 
$ f_*f^!(\un_S)$ is the properly supported stable homotopy type of $X$.
The canonical morphism $\alpha_X$ is obtained from the six-functor formalism for the stable motivic 
homotopy category $\SH(S)$, 
which implies the following fundamental properties.
\vspace{0.05in}
\begin{itemize}
\item If $X/S$ is smooth, then $f_!f^!(\un_S)=\Sigma^\infty X_+$ is the motivic suspension spectrum of $X$
\vspace{0.05in}
\item If $X/S$ is proper, then $\alpha_X$ is an isomorphism
\vspace{0.05in}
\item The morphism $\alpha_X$ is covariant with respect to proper morphisms and contravariant with respect to 
\'etale morphisms
\end{itemize}

With the intrinsic definition of $\Pi_{S}^{\infty}(X)$ in \eqref{equation:hes} we deduce a number of novel properties 
in the spirit of proper homotopy theory.
Let us fix a compactification $\bar X$ of $X$ over $S$ and denote by $\partial X$ its reduced \emph{boundary}.
Then the induced immersions $j:X \rightarrow \bar X$, $i:\partial X \rightarrow X$ form a diagram of $S$-schemes
\begin{equation}
\label{equation:compactdiagram}
\xymatrix@R=14pt@C=30pt{
X\ar@{^(->}^j[r]\ar_{f}[rd] & \bar X\ar[d]
& \partial X\ar@{_(->}_i[l]\ar^{g}[ld] \\
& S & }
\end{equation}
We observe the stable homotopy type at infinity of $X$ is determined by the data in \eqref{equation:compactdiagram} 
via a canonical equivalence
\begin{equation}
\label{equation:piinfinitycompact}
\Pi_{S}^\infty(X) 
\simeq 
g_*i^*j_*f^!(\un_S)
\end{equation}
This shows $\Pi_{S}^\infty(X)$ is independent of the chosen compactification and our construction has properties 
analogous to Deligne's vanishing cycle functor for \'etale sheaves, 
see \cite{SGA7II}.
We may reformulate \eqref{equation:piinfinitycompact} by means of the canonically induced exact homotopy sequence
\begin{equation}
\label{equation:piinfinityexact}
\Pi_{S}^\infty(X) \to
\Pi_{S}(\partial X) \oplus \Pi_{S}(X) 
\xrightarrow{i_*+j_*} 
\Pi_{S}(\bar X)
\end{equation}

By a closed pair of $S$-schemes $(Y,W)$ we mean a closed immersion $W\not\hookrightarrow Y$ of $S$-schemes, 
and a morphism $\phi\colon (Y,W)\to (X,Z)$ is an $S$-morphism $\phi\colon Y\to X$ such that $\phi^{-1}(Z)=W$.

Suppose $X$, $Y$ are proper $S$-schemes, 
and $\phi$ induces an isomorphism of formal completions
$$
\phi
\colon
\hat Y_W
\xrightarrow \cong
\hat X_Z
$$
It follows that there exists a canonical equivalence in $\SH(S)$
\begin{equation}
\label{equation:analyticinvariance}
\phi_*
\colon
\Pi_{S}^{\infty}(Y-W) 
\xrightarrow{\simeq} 
\Pi_{S}^\infty(X-Z)
\end{equation}
This shows the stable motivic homotopy type at infinity functor $\Pi_{S}^{\infty}(-)$ satisfies analyticial invariance.
The validity of \eqref{equation:analyticinvariance} follows by Wildeshaus' proof of \cite[Theorem 5.1]{Wild1} 
and \eqref{equation:piinfinityexact} (we leave further details to the interested reader).
\vspace{0.1in}

Returning to the notation in \eqref{equation:compactdiagram}, 
let us assume $\bar X$, $\partial X$ are smooth $S$-schemes, 
and write $N$ for the normal bundle of $\partial X$ in $\bar X$.
In \Cref{thm:smoothhes} we use the Euler class $e(N)$ in $\SH(S)$ to deduce the exact homotopy sequence
\begin{equation}
\label{equation:eulerhes}
\Pi_{S}^\infty(X) 
\rightarrow 
\Pi_{S}(\partial X)
\xrightarrow{e(N)} 
\Sigma^\infty\Th_{S}(N)
\end{equation}
It is helpful to think of the passage from \eqref{equation:hes} to \eqref{equation:eulerhes} in the language of 
problem solving.
Our ``problem'' is to understand $\Pi_{S}^\infty(X)$ and the ``solution'' in the smooth case is the Euler class for 
the normal bundle of the closed immersion $\partial X\not\hookrightarrow\bar X$.
\vspace{0.1in}

In the following, we further assume $\bar X$ is a smooth proper $S$-scheme and $\partial X$ is a normal crossing divisor 
on $\bar X$.
We may write $\partial X=\cup_{i \in I} \partial_i X$ as the union of its irreducible components $\partial_i X$, 
so there is a canonical closed immersion $\nu_i:\partial_i X \rightarrow \bar X$.
For any subset $J \subset I$, 
we equip $\partial_J X:=\cap_{j \in J} \partial _j X$ with its reduced subscheme structure,
where $\cap$ is suggestive notation for fiber products over the boundary $\partial X$.
If $J \subset K$, there is a canonical proper morphism $\nu_K^J:\partial_K X \rightarrow \partial_J X$.
By means of descent for the cdh-covering 
$$
\sqcup_{i\in I}\partial_i X\to \partial X
$$ 
we identify $\Pi_{S}(\partial X)$ with the 
colimit\footnote{Limits and colimits in this paper are taken in the sense of $\infty$-categories.}
of the naturally induced diagram in $\SH(S)$
\begin{equation}
\label{equation:partialXnc}
\Pi_{S}(\partial_I X)
\longrightarrow
\bigoplus_{\sharp J= \sharp I - 1} \Pi_{S}(\partial_J X)
\begin{smallmatrix}
\longrightarrow\\
\cdots \\
\longrightarrow 
\end{smallmatrix} 
\bigoplus_{\sharp J= \sharp I - 2} \Pi_{S}(\partial_J X)
\begin{smallmatrix}
\longrightarrow\\
\cdots \\
\longrightarrow 
\end{smallmatrix} 
\cdots
\begin{smallmatrix}
\longrightarrow\\
\cdots \\
\longrightarrow 
\end{smallmatrix} 
\bigoplus_{i\in I} \Pi_{S}(\partial_i X)
\end{equation} 
The face map on the summand $\Pi_{S}(\partial_K X)$ is defined by the pushforward maps
$$
\sum_{J\subset K, \sharp J= \sharp K-1} 
(\nu_{K}^J)_*
$$
Similarly, 
we identify $\Sigma^\infty\Th_{S}(N)$ with the limit of the naturally induced diagram in $\SH(S)$
\begin{equation}
\label{equation:partialXnc2}
\bigoplus_{i\in I} 
\Sigma^\infty\Th_{S}(N_i)
\begin{smallmatrix}
\longrightarrow\\
\cdots \\
\longrightarrow 
\end{smallmatrix} 
\bigoplus_{\sharp J = 2} \Sigma^\infty\Th_{S}(N_J)
\begin{smallmatrix}
\longrightarrow\\
\cdots \\
\longrightarrow 
\end{smallmatrix} 
\bigoplus_{\sharp J = 3} \Sigma^\infty\Th_{S}(N_J)
\begin{smallmatrix}
\longrightarrow\\
\cdots \\
\longrightarrow 
\end{smallmatrix} 
\cdots
\longrightarrow 
\Sigma^\infty\Th_{S}(N_I)
\end{equation}
Here, 
$N_J$ is the normal bundle of $\partial_J X$ in $\bar X$, 
and the coface map on the summand $\Sigma^\infty\Th_{S}(N_K)$ is defined by the Gysin maps
$$
\sum_{J\subset K, \sharp J= \sharp K-1} 
(\nu_{K}^J)^!
$$ 
Our general computations culminate in \Cref{thm:maincomputation}, 
where we identify $\Pi_{S}^\infty(X)$ with the homotopy fiber of the map
$$
\colim_{n \in (\Dinj)^{op}} \left(\bigoplus_{J \subset I, \sharp J=n+1} 
\Pi_{S}(\partial_JX)\right)
\xrightarrow \mu 
\underset{n \in \Dinj}\lim \left(\bigoplus_{J \subset I, \sharp J=m+1} 
\Sigma^\infty\Th_{S}(N_J)\right)
$$
induced by
$$
(\mu_{i,j})_{i,j\in I}
\colon
\bigoplus_{i \in I} \Pi_{S}(\partial_iX) 
\longrightarrow 
\bigoplus_{j \in I} \Sigma^\infty\Th_{S}(N_j)
$$
More precisely, 
$\mu_{i,j}$
is shorthand for the composite map
$$
\Pi_{S}(\partial_iX) 
\xrightarrow{\nu_{i*}}
\Pi_{S}(\bar X) 
\rightarrow
\Sigma^\infty\left(\frac{\bar X}{\bar X\smallsetminus \partial_jX}\right)
\xrightarrow \simeq
\Sigma^\infty\Th_{S}(N_j)
$$
Here, 
for the rightmost map, 
we make use of the homotopy purity equivalence
\begin{equation}
\label{equation:homotopypurity}
\frac{\bar X}{\bar X\smallsetminus \partial_jX}
\simeq
\Th_{S}(N_j)
\end{equation}
for normal bundles of closed immersions \cite[Theorem 3.2.23]{morelvoevodsky}.
\vspace{0.1in}

To refine these techniques we develop a theory of duality with compact support. 
In the process, we generalize the homotopy purity theorem and give new examples of rigid objects.
Our approach is based on the notion of a homotopically smooth morphism.
If $f:X \rightarrow S$ is a smoothable lci morphism with virtual bundle $\tau_f$ over $X$, 
we say that $f$ is \emph{homotopically smooth} (\emph{h-smooth}) if the naturally induced morphism 
$$
\mathfrak p_f:\Th(\tau_f) \rightarrow f^!(\un_S)
$$
is an isomorphism (see \Cref{df:hsmooth} for details).
Any closed immersion between smooth varieties over a field is h-smooth.
When $f$ is h-smooth and $i:Z \rightarrow X$ is a closed immersion with $Z/S$ h-smooth, 
\Cref{thm:generalizedhomotopypurity} shows the relative purity isomorphism
$$
\htp_S(X/X-Z,v) \simeq \htp_S(Z,i^*v+N_i)
$$
Here, 
$v$ is a virtual vector bundle over $X$ and $N_i$ is the (necessarily regular) normal bundle of 
$i:Z \rightarrow X$.
Under the additional assumption that $\htp_S(X,v)$ is rigid, 
we show in \Cref{cor:Atiyah} the duality with compact support isomorphism
$$
\htp_S(X,v)^\vee \simeq \htp_S^c(X,-v-\tau_f)
$$
This duality isomorphism can be seen as a motivic analog of classical topological results due to Atiyah 
\cite[\S3]{atiyah:thomcomplexes}, Milnor-Spanier \cite[Lemma 2]{milnorspanier}.
As an application,
we identify the stable motivic homotopy type at infinity of hyperplane arrangements in \Cref{subsection:smhtaioha}.
\vspace{0.1in}

In \Cref{section:Mumford}, 
we consider the case of a compactified smooth surface whose boundary is a union of rational curves. 
More precisely, 
let $X/S$ be a smooth family of surfaces over $S$ with a smooth proper compactification $\bar X/S$, 
whose boundary $\partial X$ is a normal crossing divisor on $\bar X$,	
see \Cref{df:normal_crossing}. 
We assume in addition that each branch $\partial_i X$, 
with its reduced schematic structure,
is a rational curve over $S$.
This setup gives is a quadratic generalization of Mumford's plumbing game \cite{mumfordihes} using Chow-Witt groups.
While Mumford uses orientations on the normal bundles of the branches,
which are copies of the projective line, 
much of the subtleties in our setting comes from working with twisted Milnor-Witt $K$-theory sheaves in order to 
compute the quadratic degree maps of the intersections of the branches taking values in the Grothendieck-Witt ring.
We decompose $\htp_S(\partial X)$ into its combinatorial part $\mathcal D_X$, 
which is an Artin motivic spectrum over $S$, 
and its geometrical part $\oplus_{i \in I} \un_S(1)[2]$.
In \Cref{thm:smhatoomatrix}, 
we show $\htp^\infty_S(X)$ is the homotopy fiber of a map 
$$
\mathcal D_X \oplus \bigoplus_{i \in I} \un_S(1)[2] 
\xrightarrow{
\begin{pmatrix}
a & b \\
b' & \mu
\end{pmatrix}}
\mathcal D'_X \oplus \bigoplus_{j \in I} \Th(N_j^0)
$$
The map $\mu$ is computed by a 
``quadratic Mumford matrix" defined by the composite
\begin{equation}
\mu_{ij}:
\un_S(1)[2]
\hookrightarrow \htp_S(\partial_{i} X) 
\xrightarrow{\nu_{i*}} \htp_S(\bar X) 
\xrightarrow{\nu_{j}^!} \htp_S(\partial_{j} X,\twist{N_j})
\twoheadrightarrow
\Th(N_j^0) 
\end{equation}

When $S=\Spec(\mathcal O)$ is a semi-local essentially smooth scheme of a field,
we can interpret $\mu_{ij}$ as the class of a quadratic form $(\partial_i X,\partial_j X)_{quad} \in \GW(\mathcal O)$
called the {\it quadratic degree} of the intersections of the divisors $\partial_iX$ and $\partial_j X$.
The close connection with quadratic forms arises since elements of the $i$th Chow-Witt group is represented by 
formal sums of subvarieties $Z$ of codimenison $i$ equipped with an element of $\GW(k(Z))$.
Moreover, 
the rank of the quadratic degree equals the corresponding Mumford degree.
\vspace{0.1in}

Further we specialize our results to motives.
The above formula implies that the motive at infinity, 
i.e., 
Wildeshaus' boundary motive,
of $X$ is an Artin-Tate motive.
When $S$ is a finite field, a global field, or a number ring, 
we have the motivic $t$-structure on rational Artin-Tate motives at our disposal
(see \cite{LevineAT} for the case of fields, and \cite{ScholbachAT} for number rings).
Our main result can be restated (see \Cref{cor:Mumford_ATmotives}) in terms of abelian Artin-Tate mixed motives, 
by asserting the existence of an exact sequence 
\begin{align*}
0 \rightarrow &\HM_3^\infty(X) \rightarrow \bigoplus_{i \in I} \un_S(2)
\xrightarrow{\sum_{i<j} p_{ij}^{i!}-p_{ij}^{j!}} \bigoplus_{i<j} M_S(\partial_{ij} X)(2)  \\
& \rightarrow \HM_2^\infty(X) \rightarrow \bigoplus_{i \in I} \un_S(1) 
\xrightarrow{\ \mu\ } \bigoplus_{j \in I} \un_S(1) \\
& \rightarrow \HM_1^\infty(X) \rightarrow \bigoplus_{i<j} M_S(\partial_{ij} X)
\xrightarrow{\sum_{i<j} p^i_{ij*}-p^j_{ij*}} \bigoplus_{i \in I} \un_S 
\rightarrow \HM_0^\infty(X) \rightarrow 0
\end{align*}
Here, 
$\mu$ is the Mumford 
matrix\footnote{Due to the hypothesis on $S$, one has $\Hom(\un_S(1),\un_S(1))\cong\QQ$.}
and $\HM_i^\infty(X)$ is the $i$th homology motive at infinity of $X$.
One can extract the weight filtration of explicit examples of abelian mixed motives from the above exact sequence.
This gives an abelian motivic extension of Mumford's work to the arithmetic case.
Due to its motivic nature, 
it has the advantage of giving universal formulas in the various realizations
($\ell$-adic, rigid, syntomic,...) of motives. 
For example, 
this informs us about Galois representations attached to $\HM_i(X)$.
\vspace{0.1in}

We illustrate the general theory with examples of $\mathbf{A}^{1}$-equivalent smooth affine surfaces with non-isomorphic 
stable motivic homotopy types at infinity.
For any integer $n>0$, 
the Danielewski surface $D_{n}$ is the closed subscheme of $\mathbf{A}^{3}$ cut out by the equation $x^{n}z=y(y-1)$, 
see \cite{Danielewskisurfaces}.
We note that $D_{1}$ is the Jouanolou device over $\mathbf{P}^{1}$;
in fact, 
$D_{n}$ is $\mathbf{A}^{1}$-equivalent to $\mathbf{P}^{1}$ for all $n$, see \cite[\S 3.4]{asokostvar}.
Over any field $k$, 
one can distinguish between $\Pi_{k}^{\infty}(D_{m})$ and $\Pi_{k}^{\infty}(D_{n})$ for $m\neq n$
by viewing Danielewski surfaces as affine modifications of $\mathbf{A}^{2}$.
We refer to \Cref{subsection:examples} for precise statements and further examples, 
\cite{zbMATH07149737} for background on $\mathbf{A}^{1}$-contractibility of affine modifications, 
and \cite{fieseler} for Fieseler's work on first homology at infinity of Danielewski surfaces over 
the complex numbers.
The technique of affine modifications gives an affirmative answer to Problem 3.4.5 in \cite{asokostvar}.
Our viewpoint, 
however, 
is that $\Pi_{k}^{\infty}$ is the first step towards a refined invariant in 
unstable motivic homotopy theory, 
see \Cref{ex:adA1contractible}.
The problem of defining a notion of unstable motivic homotopy types at infinity captures the tension 
between unstable and stable motivic homotopy theory.
For example, 
the six functor formalism is not available in the unstable setting.
To remedy this one can take into account all possible smooth compactifications.
Nonetheless some of the techniques developed in this paper will carry over to unstable motivic homotopy categories, 
e.g., 
the calculations in \Cref{sec:crossing_sing} hold in the cdh-topology, 
and one can expect more developments along these lines.

\begin{rem}
\label{remark:generality}
This paper's results hold more generally for any motivic $\infty$-category such as 
triangulated and abelian mixed motives, Artin-Tate motives,
\'etale motives, torsion and $\ell$-adic categories, mixed Hodge modules, etc., in place of $\SH$. 
In the case there exists a realization functor that commutes with the six operations, e.g., 
the Betti or $\ell$-adic realizations, then this follows from the universality of $\SH$.
\end{rem}

\subsubsection{Conventions}
Our results are couched in the axiomatic setting of \cite{CD3}, \cite{Khan} which complements \cite{Ayoub}.
We fix a \emph{motivic $\infty$-category} (\cite[Definition 2.4.45]{CD3}) $\T$ over the category of qcqs schemes, 
i.e., 
a \emph{monoidal stable homotopy functor} according to \cite{Ayoub}. 
Our primary example is the motivic stable homotopy category $\SH$.
In the language of presentable stable monoidal $\infty$-categories \cite{Khan}, 
$\SH$ is the initial motivic $\infty$-category.
Thus there is a unique morphism of motivic $\infty$-categories $\SH \rightarrow \T$.
To maintain intuition, we shall refer to the objects of $\T(S)$ as $\T$-spectra over $S$.

\subsubsection{Acknowledgements}

We gratefully acknowledge the support of the Centre for Advanced Study at the Norwegian Academy of Science and Letters in Oslo,
Norway, which funded and hosted our research project ``Motivic Geometry" during the 2020/21 academic year,
and we extend our thanks to the French ``Investissements dAvenir" project ISITE-BFC (ANR-lS-IDEX-OOOB), 
the French ANR project ``FIBALGA" (ANR-18-CE40-0003),
and the RCN Frontier Research Group Project no. 250399 ``Motivic Hopf Equations." 
{\O}stv{\ae}r acknowledges the generous support from Alexander von Humboldt Foundation and 
The Radboud Excellence Initiative.

\section{Preliminaries and notation}
\label{sec:notations}

We work over quasi-coherent and quasi-compact base schemes $S$.
The natural framework for this paper is Morel-Voevodsky's stable homotopy category $\SH(S)$ over $S$.
Owing to the works \cite{Ayoub,Ayoub2}, \cite{CD3}, 
for varying $S$,  
these categories satisfy \emph{Grothendieck' six functor formalism} which we will use extensively.
The elimination of the noetherian hypothesis was achieved in \cite[Appendix C]{hoyoisquadratic}.
We will freely use the terminology and notations of \cite{CD3} (see Section A.5 for a summary).
In particular, 
$\SH(S)$ is a triangulated category and a closed symmetric monoidal category;
moreover, 
the two structures are compatible, 
see e.g., 
\cite{BRTV}, \cite{DRO}, \cite{hoyoisquadratic}, \cite{JardineMSS}, \cite{voevodskyicm}.
\vspace{0.1in}

In \Cref{section:Mumford} we employ the fact that $\SH(S)$ admits a stable $\infty$-categorical enhancement.
This means that the six functors can be viewed as $\infty$-functors, and commutative diagrams involving Grothendieck's 
six functor formalism can be enhanced into homotopy coherent diagrams.
This useful point of view started, 
to our knowledge, 
with the paper \cite{LiuZheng} on \'etale $\ell$-adic sheaves.
Following Lurie's idea, systematic treatment is given in \cite[Chapter 3, Appendix]{GaitRoz}. 
Within the framework of motivic categories,
this was exploited in Khan's thesis \cite{Khan}. 
For a recent reference, 
we refer to \cite[Appendix A]{BRTV}.

As noted in \Cref{remark:generality}, 
our main results are also valid in arbitrary triangulated motivic $\infty$-categories such as the following examples.
\begin{itemize}
\item $\DM_\QQ$ the $\infty$-category of rational mixed motives: see \cite{CD3}
\item $\DM$ the $\infty$-category of integral motives: 
we take as a model the category of modules over Spitzweck motivic cohomology ring spectrum relative to $\ZZ$: 
see \cite{SpiMod}\footnote{This viewpoint was first advocated in \cite{orpaomodulescras,orpaomodules}.
If one restricts to schemes over a prime field $k$ and inverts the characteristic exponent of $k$, 
one can take the so-called $\cdh$-motives as defined in \cite{CD4} (using $\cdh$-sheaves with transfers).}
\item $\DM_\et$ the $\infty$-category of \'etale mixed motives: see \cite{AyoubEt,CD5}
\item $D^b_{c}(-_\et,\ZZ_\ell)$ the $\infty$-category of integral $\ell$-adic \'etale sheaves: 
see \cite{BBD}, \cite[7.2.18]{CD5}
\item $\mathrm{DH}$ the $\infty$-category of motivic Hodge modules,
which should correspond to complexes of Saito's mixed Hodge modules of geometric origin
(obtained by realization of mixed motives): see \cite{Drew}
\end{itemize}

Recall that an object $M$ of a symmetric monoidal category $\mathcal M$ is \emph{rigid} 
(or \emph{strongly dualizable}) with dual $N$ if the functor $(M \otimes -)$ is left and right adjoint to 
$(N \otimes -)$.\footnote{This notion was introduced in \cite{DP},
but it appeared earlier in the theory of pure motives \cite[Prop. 4]{Dem} as a key property in the 
theory of Tannakian categories \cite{Saa}. 
We choose to write rigid since it is shorter than strongly dualizable. 
Note that \emph{dualizable} is ambiguous as it could also refer to the property $DD(M) \simeq M$ for a dualizing 
functor $D$ in Grothendieck-style duality. 
Recall, for example, that for spectra over a scheme of characteristic $0$, the latter property amounts to constructibility  
(see \cite{Ayoub}).}
It is well-known that for $X/S$ smooth and proper, 
with tangent bundle $T_X$,
the spectrum $\Sigma^\infty X_+$ over $S$ is rigid with dual $\Sigma^\infty \Th(-T_X)$, 
see e.g., \cite[2.4.31]{CD3}.\footnote{This is in fact an easy consequence of the six functor formalism, 
as formulated in \cite[Theorem 2.4.50]{CD3}, given that $\Sigma^\infty X_+ \simeq f_!f^!(\un_S)$, 
where $f:X \rightarrow S$ is the projection map:
indeed, one deduces from the projection formula and the purity property that $f_!f^!(\un_S)$ is rigid with dual 
$f_*(\un_X) \simeq f_\sharp(\Th_X(-T_X))$.}

\section{Stable homotopy at infinity}
\label{section:shai}

\subsection{Main definitions and examples}

This section defines the stable motivic homotopy type at infinity for separated schemes defined over a 
quasi-compact quasi-separated (=qcqs) base scheme $S$.
We begin with a ``classical'' definition of stable motivic homotopy types,
based on the analogy with Voevodsky's theory of motivic complexes (\cite[Chapter 5]{FSV}).

\begin{df}
\label{df:smhtdef}
Let $f\colon X\to S$ be a separated morphism.
The stable motivic homotopy type $\htp_{S}(X)$ and the properly supported stable motivic homotopy type 
$\htp_{S}^c(X)$ of $X$ in $\SH(S)$ are defined by 
\begin{align*}
\htp_{S}(X)&:=f_!f^!(\un_S) \\
\htp_{S}^c(X)&:=f_*f^!(\un_S)
\end{align*}
\end{df}
The exchange natural transformation $\alpha_f:f_! \rightarrow f_*$ induces a canonical map
$$
\alpha_X:\htp_{S}(X) \rightarrow \htp_{S}^c(X)
$$

\begin{rem}
Suppose $k$ is a field. 
The corresponding cohomological notation for the motivic complex (resp. properly supported cohomological complex) 
of $X/k$ is $f_*(\un_X)$ (resp. $f_!(\un_X)$). 
\Cref{df:smhtdef} complies with more classical invariants and allows for a duality theory.
\end{rem}

\begin{num}
The standard properties of the six functor formalism for $\SH(S)$ implies:
\begin{itemize}
\item If $X/S$ is smooth then $\htp_{S}(X)=\Sigma^\infty X_+$
\item If $X/S$ is proper then $\alpha_X$ is an isomorphism
\item $\htp_{S}(-)$ (resp. $\htp_{S}^c(-)$) is covariant functorial for all morphisms (resp. proper morphisms) 
and contravariant functorial with respect to \'etale morphisms
\item $\alpha_X$ is natural with respect to proper covariance and \'etale contravariance
\end{itemize}
\end{num}


\begin{df}
\label{df:smhtatinfinitydef}
Let $f\colon X\to S$ be a separated morphism.
The \emph{stable homotopy type at infinity} of $X/S$ is the homotopy fiber of $\alpha_X$ so that there is an 
exact homotopy sequence
$$
\htp_S^\infty(X) \rightarrow \htp_S(X) \xrightarrow{\alpha_X} \htp_S^c(X)
$$
\end{df}

\begin{ex}\textit{Motivic realization}. 
\label{ex:realization}
Suppose $S$ is a perfect field $k$ with characteristic exponent $p$.
In this case, 
see \cite{mhskpao}, \cite{orpaomodules}, 
there is a realization functor to motives
\begin{equation}
\label{equation:motivicrealization}
\SH(k) \rightarrow \DM(k)[1/p]
\end{equation}
If $X\rightarrow \Spec(k)$ is smooth, 
the motivic realization functor sends $\htp_k(X)$ to $M(X)$, 
Voevodsky's homological motive of $X$,
and it sends $\htp_k^c(X)$ to $M^c(X)$, 
Voevodsky's homological motive of $X$ with compact support (see \cite[Proposition 8.10]{CD5}).
It follows that the motivic realization functor sends $\htp_k^\infty(X)$ to \emph{the boundary motive} 
$\partial M(X)$ (see Wildeshaus \cite{Wild1}).
In particular, 
the Betti or $\ell$-adic cohomology of $\htp_k^\infty(X)$ coincides with the so-called \emph{interior cohomology} of $X$.

We generalize the above discussion to arbitrary base schemes in \Cref{section:Mumford}.
\end{ex}

\begin{ex}\textit{Betti realization}. 
\label{ex:realization2}
If the field $k$ admits a complex embedding, 
there is a Betti realization functor to the stable homotopy category
\begin{equation}
\label{equation:bettirealization}
\SH(k) \rightarrow \SH^{top}
\end{equation}
Owing to Ayoub's enhancement of \eqref{equation:bettirealization} to an arbitrary base scheme using the technique 
of analytical sheaves \cite{Ayoub3}, 
one derives that $\htp_S(X)$ realizes to the singular chain complex $S_*(X)$, and $\htp_S^c(X)$ realizes to the 
Borel-Moore singular chain complex $S_*^{BM}(X)$.
\Cref{ex:realization} implies the Betti realization of $\htp_S^\infty(X)$ coincides with the singular complex at infinity, 
$S_*^\infty(X)$, 
whose homology groups are the singular homology groups at infinity $H_*^\infty(X)$.
Indeed, 
for a topological space $W$, 
there is a distinguished triangle of chain complexes of abelian groups
\begin{equation}
\label{eq:top_infty_hlg}
S^\infty_*(W) 
\rightarrow 
S_*(W) 
\xrightarrow{\alpha_{W}} 
S_*^{lf}(W)
\rightarrow 
S^\infty_*(W)[1] 
\end{equation}
Here, 
$S_*(W)$ (resp. $S_*^{lf}(W)$) be the complex of singular chains (resp. locally finite singular chains) on $W$,
see \cite[Chapter 3]{HR}.
If $W$ is locally contractible and $\sigma$-compact,
then $S_*^{lf}(W)$ is quasi-isomorphic to the Borel-Moore complex of $W$.
\end{ex}

\subsection{Computation via compactifications}

\begin{num}
\label{num:thp-infty&compact}
Let $f\colon X\to S$ be a separated morphism.
We fix an arbitrary compactification $\bar X$ of $X$ over $S$ and denote by $\partial X=(\bar X-X)_{red}$ its \emph{boundary}.
Consider the immersions $j:X \rightarrow \bar X$, $i:\partial X \rightarrow X$.  
Using the localization property of $\SH$, 
one derives the exact homotopy sequence
\begin{equation}
\label{equation:localizationhes}
\htp_{S}^c(\partial X) \xrightarrow{i_*} \htp_{S}^c(\bar X) \xrightarrow{j^*} \htp_{S}^c(X)
\end{equation}
Moreover, 
using properness of $\bar X$ and $\partial X$, 
there is a naturally induced commutative diagram with vertical isomorphisms
\begin{equation}
\label{equation:localizationcd}
\xymatrix@=14pt{
\htp_{S}(\partial X)\ar^{i_*}[r]\ar_{\alpha_{\partial X}}^\sim[d]
 & \htp_{S}(\bar X)\ar^{\alpha_{\bar X}}_\sim[d] \\
\htp_{S}^c(\partial X)\ar^{i_*}[r] & \htp_{S}^c(\bar X)
}
\end{equation}

We define the relative stable motivic homotopy type $\htp_{S}(\bar X,X)$ of the pair $(\bar X, X)$ by the exact homotopy sequence
\begin{equation}
\label{equation:pairhes}
\htp_{S}(X) \xrightarrow{j_*} \htp_{S}(\bar X) \xrightarrow \pi \htp_{S}(\bar X,X)
\end{equation}
\end{num}

In the next result, directly inspired by \cite[Theorem 1.6]{Wild2},
we relate \eqref{equation:localizationhes} and \eqref{equation:pairhes} 
to the stable motivic homotopy type at infinity.

\begin{prop}
\label{prop:basic_comput_thp-infty}
There exist canonical maps $\beta$ and $\delta$ in $\SH(S)$ rendering the following diagram of homotopy exact columns 
and rows commutative
\begin{equation}
\label{equation:3by3diagram}
\xymatrix@=20pt
{
0\ar[r]\ar[d] & \htp_{S}(\partial X)\ar^{i_*}[d] \ar@{=}[r] & \htp_{S}(\partial X)\ar^{\beta}[d] \\
\htp_{S}(X)\ar^{j_*}[r]\ar@{=}[d] & \htp_{S}(\bar X)\ar^{j^*}[d]\ar^-{\pi}[r] & \htp_{S}(\bar X,X)\ar^{\delta}[d] \\
\htp_{S}(X)\ar^{\alpha_X}[r] & \htp_{S}^c(X)\ar[r] & \htp_{S}^\infty(X)[1]
}
\end{equation}
The maps $\beta$ and $\delta$ are natural in $(\bar X,X,\partial X)$, 
covariantly functorial with respect to proper maps, 
and contravariantly functorial with respect to \'etale maps (see \ref{num:fpremark}).
\end{prop}
\begin{proof}
To define $\beta$ and $\delta$ canonically we rely on the six functors formalism.
By the localization property, 
there exist exact homotopy sequences of functors
\begin{align}
\label{eq:loc1}
i_*i^! &\xrightarrow{ad'_i} 1 \xrightarrow{ad_j} j_*j^* \\
\label{eq:loc2}
j_!j^* &\xrightarrow{ad'_j} 1 \xrightarrow{ad_i} i_*i^*
\end{align}
A combination of \eqref{eq:loc1} and \eqref{eq:loc2} yields the commutative diagram
\begin{equation}
\label{equation:3by3diagramdefinition}
\xymatrix@=34pt{
0\ar[r]\ar[d] & i_*i^!\ar@{=}[r]\ar_{ad'_i}[d] & i_*i^!\ar^{i_*i^*.ad'_i}[d] \\
j_!j^*\ar^{ad'_j}[r]\ar@{=}[d] & 1\ar^{ad_i}[r]\ar_{ad_j}[d] & i_*i^*\ar^{i_*i^*.ad_j}[d] \\
j_!j^*\ar^{ad'_j.j_*j^*}[r] & j_*j^*\ar^{ad_i.j_*j^*}[r] & i_*i^*j_*j^*
}
\end{equation}
Here we us the fact that $i_*$ and $j_*$ are fully faithful, 
so that $j_!j^*j_*j^* \simeq j_!j ^*$ and $i_*i^*i_*i^! \simeq i_*i^!$.
The bottom row is obtained from \eqref{eq:loc2} by post-composition with $j_*j^*$, 
and the rightmost column is obtained from \eqref{eq:loc1} by pre-composition with $i_*i^*$.
Thus, all columns and rows in \eqref{equation:3by3diagramdefinition} are homotopy exact.
Let $\bar f:\bar X \rightarrow S$ be the structure map.
Inserting $\bar f^!(\un_S)$ into \eqref{equation:3by3diagramdefinition} and applying $\bar f_!=\bar f_*$ 
yields the desired diagram \eqref{equation:3by3diagram}. 
\end{proof}

\begin{rem}
\label{rem:pi_infty&functoriality}
In the setting of \Cref{prop:basic_comput_thp-infty} there are structure maps
$$
\xymatrix@R=14pt@C=30pt{
X\ar@{^(->}^j[r]\ar_f[rd] & \bar X\ar^{\bar f}[d]
& \partial X\ar@{_(->}_i[l]\ar^g[ld] \\
& S &
}
$$
The proof of \Cref{prop:basic_comput_thp-infty} gives the following diagram of exact homotopy sequence
$$
\xymatrix@C=30pt@R=14pt{
\htp_{S}(\partial X)\ar^-{\beta}[r]\ar_\sim[d] & \htp_{S}(\bar X,X)\ar^-{\delta}[r]\ar|\sim[d]
& \htp_{S}^\infty(X)[1]\ar^\sim[d] \\
g_!g^!(\un_S)\ar[r]\ar@{=}[d]
& g_!i^*\bar f^!(\un_S)\ar[r]\ar@{=}[d]
& g_*i^*j_*f^!(\un_S)\ar@{=}[d] \\
\bar f_!i_*i^!\bar f^!(\un_S)\ar^{ad'_i}[r]
& \bar f_!i_*i^*\bar f^!(\un_S)\ar^{ad_j}[r] & \bar f_!i_*i^*j_*j^*\bar f^!(\un_S)
}
$$
That is,
$\beta$ (resp. $\delta$) is induced by the counit map $ad'_i$ (resp. unit map $ad_j$) of the adjunction $(i_*,i*)$ 
(resp. $(j^*,j_*)$).
Here we use simplified notation for the morphisms in the lower rows. 
In particular, 
we obtain the following interesting formula for the stable homotopy type at infinity of $X$
\begin{equation}
\htp_{S}^\infty(X) \simeq g_*i^*j_*f^!(\un_S)[-1]
\end{equation}
This shows $\htp_{S}^\infty(-)$ is analogous to the vanishing cycle functor for \'etale sheaves \cite{SGA7II}.
\end{rem}

\begin{num}
\label{num:fpremark}
\textit{Functoriality properties}.
Suppose
$$
\xymatrix@R=14pt@C=20pt{
Y\ar@{^(->}^l[r]\ar_f[d] & \bar Y\ar^{\bar f}[d]
& \partial Y\ar@{_(->}_k[l]\ar^g[d] \\
X\ar@{^(->}^j[r] & \bar X & \partial X \ar@{_(->}_i[l]
}
$$
is a commutative diagram of $S$-schemes, 
where $\bar X$, $\bar Y$ are proper, $i,k$ closed immersions, and $j,l$ open immersions.
Then, if $f$ is proper, there is a commutative diagram
$$
\xymatrix@C=30pt@R=18pt{
\htp_{S}(\partial Y)\ar^{\beta_Y}[r]\ar_{g_*}[d]
& \htp_{S}(\bar Y,Y)\ar^{\delta_Y}[r]\ar|{(\bar f,f)_*}[d]
& \htp_{S}^\infty(Y)[1]\ar^{f_*}[d] \\
\htp_{S}(\partial X)\ar^{\beta_X}[r] & \htp_{S}(\bar X,X)\ar^{\delta_X}[r]
& \htp_{S}^\infty(X)[1]
}
$$
Moreover, if $\bar f$, $f$, $g$ are \'etale, we obtain the commutative diagram
$$
\xymatrix@C=30pt@R=18pt{
\htp_{S}(\partial X)\ar^{\beta_X}[r]\ar_{g^*}[d]
& \htp_{S}(\bar X,X)\ar^{\delta_X}[r]\ar|{(\bar f,f)^*}[d]
& \htp_{S}^\infty(X)[1]\ar^{f^*}[d] \\
\htp_{S}(\partial Y)\ar^{\beta_Y}[r] & \htp_{S}(\bar Y,Y)\ar^{\delta_Y}[r]
& \htp_{S}^\infty(Y)[1]
}
$$
\end{num}

\begin{rem}
Another way of stating \Cref{prop:basic_comput_thp-infty} is that $\htp_{S}^\infty(X)$ 
is the homotopy fiber of the canonical map
\begin{equation}
\label{equation:smhthf}
\htp_{S}(\partial X) \oplus \htp_{S}(X) \xrightarrow{i_*+j_*} \htp_{S}(\bar X)
\end{equation}
Under motivic realization, 
\eqref{equation:smhthf} becomes the fundamental formula for the boundary motive in \cite[Proposition 2.4]{Wild1}.
Following the proof method of  \cite[Theorem 5.1]{Wild1},
we deduce the theorem of analyticial invariance \eqref{equation:analyticinvariance} for $\htp_{S}^{\infty}(-)$.
\end{rem}

\subsection{The case of smooth compactifications}
\label{section:tsc}

\begin{num}
\textit{Abstract Euler classes}.
Recall that for a vector bundle $V$ over a scheme $Z$, 
the (abstract) Euler class\footnote{It can be interpreted as a twisted cohomology class $e(V) \in H^0(Z,\langle V \rangle)$, 
in the $0$-th stable cohomology group, 
which generalizes the well-known Euler classes in oriented and $\SL^c$-oriented cohomology theories.} 
$e(V)$ defined in \cite{DJK} is the map in $\SH(Z)$ 
$$
e(V):\un_Z\simeq \Th(0_X) \xrightarrow{s_*} \Th_Z(V)
$$
induced by the zero section $s$ of $V$.
If $g:Z \rightarrow S$ is smooth, 
the left adjoint $g_\sharp$ to $g^{\ast}$ yields by abuse of notation a map
$$
e(V):\htp_{S}(Z) \rightarrow \Th(V)
$$
in $\SH(S)$.  
We develop a technique for computing stable motivic homotopy types at infinity involving abstract Euler classes 
for normal bundles.
\end{num}

\begin{thm}
\label{thm:smoothhes}
Let $X$ be an $S$-scheme with compactification $\bar X$ and boundary $\partial X=\bar X-X$.
Suppose $\bar X$, $\partial X$ are smooth $S$-schemes, 
and let $N$ denote the normal bundle of $\partial X$ in $\bar X$.
Then there exists an exact  homotopy sequence in $\SH(S)$
\begin{equation}
\label{equation:smoothhes}
\htp_{S}^\infty(X) 
\rightarrow \htp_{S}(\partial X)
\xrightarrow{e(N)} \Th(N)
\end{equation}
\end{thm}
\begin{proof}
The exact homotopy sequence \eqref{equation:smoothhes} follows from Proposition \ref{prop:basic_comput_thp-infty} and an 
analysis of $\beta_{X}\colon\htp_{S}(\partial X) \rightarrow \htp_{S}(\bar X,X)\simeq \Sigma^\infty(\bar X/\bar X-\partial X)$.
Owing to \cite{DJK}, the latter map coincides with the composite
\begin{equation}
\label{equation:smoothhes2}
\Sigma^\infty \partial X_+ \xrightarrow{i_*} \Sigma ^\infty \bar X_+
\rightarrow \Sigma^\infty(\bar X/\bar X-\partial X)
\end{equation}
Homotopy purity for closed embeddings \cite[Theorem 3.2.23]{morelvoevodsky} identifies \eqref{equation:smoothhes2} with 
\begin{equation}
\label{equation:smoothhes3}
\Sigma^\infty \partial X_+ 
\xrightarrow{i_*} \Sigma ^\infty \bar X_+
\xrightarrow{i^*} \Th(N)
\end{equation}
The excess intersection formula in \cite[3.2.6]{DJK} equates the composite $i^*i_*$ in \eqref{equation:smoothhes3} with the 
abstract Euler class $e(N)$.
\end{proof}

\begin{cor}
Under the assumptions in \Cref{thm:smoothhes} there exists a canonical isomorphism
$$
\htp_{S}^{\infty}(X) \simeq \htp_{S}(N^\times)
$$
Here, 
$N^\times$ is the complement of the zero section of the normal bundle $N$.
If $N$ contains a trivial direct factor,
or more generally the Euler class $e(N)=0$,
then 
$$
\htp_{S}^{\infty}(X) \simeq \htp_{S}(\partial X) \oplus \Th(N)[-1]
$$
Therefore, if $N$ is trivial of rank $d$, we have
$$
\htp_{S}^{\infty}(X) \simeq \htp_{S}(\partial X) \oplus \un(d)[2d-1]
$$
\end{cor}

\begin{ex}
\label{ex:affinespace}
For $d>0$ the stable homotopy type at infinity of the affine space $\AA^{d}$ is given by 
$$
\htp_{S}^{\infty}(\AA^{d}) \simeq \un \oplus \un(d)[2d-1]
$$
This follows since $\AA^{d}$'s Euler class vanishes.
\end{ex}

\begin{rem}
We note that $\htp_k^{\infty}(X)$ is a strictly finer invariant than its motivic realization, 
the boundary motive $\partial M(X)$ in Example \ref{ex:realization}.
Indeed, \Cref{thm:smoothhes} shows there is an exact homotopy sequence in $\DM(k)[1/p]$
$$
\partial M(X) \rightarrow M(\partial X)
\xrightarrow{\tilde c_r(N)} M(\partial X)(r)[2r]
$$
Here, 
$r$ is the rank of $N$ and the map $\tilde c_r(N)$ is induced by multiplication with the top Chern class 
$c_r(N) \in \mathrm{CH}^r(\partial X) \simeq  \Hom(M(\partial X),\un(r)[2r])$.
In particular, 
the sequence splits if $c_r(N)=0$.
However, 
the vanishing of the homotopy Euler class $e(N)$, 
which implies the vanishing of the Euler class in Chow-Witt groups, 
is a strictly stronger condition than the vanishing of the top Chern class.
For the smooth affine quadric $5$-fold $Q\colon x_1y_1+x_2y_2+x_3y_3=1$ in $\AA^{6}$, 
the kernel of the surjection $(x_1,x_2,x_3): k[Q]^{3}\rightarrow k[Q]$ defines a nontrivial and 
stably trivial vector bundle $\xi$ of rank $2$ on $Q$.
Since $\det(\xi)$ is trivial, $\xi$ is orientable. 
While $\xi$'s Chern classes are trivial, 
$\xi$'s Euler class in $\CHt^{2}(Q)=\mathrm{K}_{-1}^{\mathrm{MW}}(k)$ equals $\eta$, 
see the case $n=2$ in \cite[Lemma 3.5]{zbMATH06349725}. 
\end{rem}

\section{Duality and fundamental class of the diagonal}
\label{section:Atiyah}

In this section, 
we develop computational techniques based on the notion of homotopically smooth morphisms 
with respect to a motivic $\infty$-category.
Our method involves duality with compact support instead of compactifications, 
as in \Cref{section:shai}.
In the process, we generalize the Morel-Voevodsky homotopy purity theorem \cite{morelvoevodsky} and give 
new examples of rigid objects;
see \cite[Remark 8.2]{motiviclandweber} for examples of compact non-rigid objects.

\subsection{Homotopical smoothness and generalized purity}

\begin{num}\label{num:Thom_sapces}
In the setting of stable motivic homotopy,
we shall use the technique of twisted Thom spaces to deduce the duality results we use in our computations.
Recall the Thom space functor, 
from vector bundles over a scheme $X$ to the pointed homotopy category over $X$, 
is given by
$$
\Th:E \mapsto \Sigma^\infty \Th(E), \Th(E)=E/E^\times
$$
It sends direct sums to tensor products. 
Owing to \cite[Remark 2.4.15]{CD3} this functor admits an extension to a monoidal functor,
called the Thom spectrum functor
$$
\Th:\uK(X) \rightarrow (\Pic(\SH(X)),\otimes)
$$
from the Picard category of virtual bundles over $X$ to the monoidal category of $\otimes$-invertible objects in $\SH(X)$.
When $v=\langle E \rangle$ is the virtual bundle associated to a vector bundle, 
one has $\Th(v)=\Sigma^\infty \Th(E)$.
Given a smooth morphism $p:X \rightarrow S$ it is sometimes convenient to simply ``forget the base" of Thom spectra;
we set $\Th_S(v):=p_\sharp \Th(v)$.

Using this observation, we extend \Cref{df:smhtdef} as follows.
\end{num}
\begin{df}\label{df:internal_theories}
Let $f:X \rightarrow S$ be a separable morphism and $v$ a virtual vector bundle over $X$.
One associates to $X/S$ and $v$ the following motivic spectra over $S$:
\begin{itemize}
\item Homotopy: $\htp_S(X,v)=f_!(\Th(v) \otimes f^!(\un_S))$
\item Cohomotopy: $\cohtp_S(X,v)=f_*(\Th(v) \otimes f^*(\un_S)) \simeq f_*(\Th(v))$
\item Borel-Moore (or (co)compactly supported) homotopy: $\htp^c_S(X,v)=f_*(\Th(v) \otimes f^!(\un_S))$
\item Compactly supported cohomotopy: $\cohtp^c_S(X,v)=f_!(\Th(v) \otimes f^*(\un_S)) \simeq f_!(\Th(v))$
\end{itemize}
\end{df}

\Cref{df:internal_theories} specializes to \Cref{df:smhtdef} in the case of the trivial vector bundle $v=0$, 
i.e., 
we have $\htp_S(X,0)=\htp_S(X)$ and $\htp^c_S(X,0)=\htp^c_S(X)$.
As in the previous case, 
the natural transformation $\alpha_f:f_! \rightarrow f_*$ yields canonical maps:
\begin{itemize}
\item $\alpha_{X/S}:\htp_S(X,v) \rightarrow \htp^c_S(X,v)$
\item $\alpha'_{X/S}:\cohtp^c_S(X,v) \rightarrow \cohtp_S(X,v)$ (``forgetting compact support")
\end{itemize}
Both $\alpha_{X/S}$ and $\alpha'_{X/S}$ are isomorphisms whenever $X/S$ is proper.

\begin{rem}
\Cref{df:internal_theories} can be adapted to an arbitrary triangulated motivic category $\T$ in the sense of \cite{CD3}. 
This generality is not strictly needed if there exists a realization functor $\SH(S) \rightarrow \T$ commuting with the 
six operations, 
as in the case of motives $\DM_{t}(S)$.
Twisting by vector bundles is not pertinent for motives since there exists well-behaved Thom isomorphisms 
$\Th(v) \simeq \un_S(r)[2r]$, 
where $r$ is the rank of $v$.\footnote{On the contrary, 
in motivic homotopy theory the consideration of twists is indispensable.}
The four theories in \Cref{df:internal_theories} realize to the homological motive $M_S(X)$,
the cohomological motive $h_S(X)$ (``Chow motive" when $X/S$ is smooth proper),
the ``homological motive with compact support" $M^c_S(X)$ in Voevodsky's terminology, 
and the cohomological motive with compact support $h_S^c(X)$, 
respectively.
\end{rem}

\begin{num}
\label{num:natural_fct} 
\textit{Natural functoriality}:
Let $f:Y \rightarrow X$ be a separable morphism of $S$-schemes.
Our choice of terminology can be explained by the following naturally induced maps:
\begin{itemize}
\item $f_*:\htp_S(Y,f^{-1}v) \rightarrow \htp_S(X,v)$
\item $f^*:\cohtp_S(X,v) \rightarrow \cohtp_S(Y,f^{-1}v)$
\item $f_*:\htp^c_S(Y,f^{-1}v) \rightarrow \htp^c_S(X,v)$
\item $f^*:\cohtp_S(X,v) \rightarrow \cohtp_S(Y,f^{-1}v)$
\end{itemize}
In addition, 
when $f$ is proper then the comparison maps $\alpha_{X/S}$ and $\alpha'_{X/S}$ are compatible with 
$f_*$, $f^*$, $f_*$, $f^*$.
\end{num}

\begin{rem}\label{rem:functoriality}
Note, 
in particular, 
that homotopy twisted by some virtual bundle $w$ on $Y$,
with or without compact support, 
is not covariant functorial unless $w$ is the pullback of some virtual bundle on $X$.
\end{rem}

\begin{num}
\label{num:Gysin} 
\textit{Exceptional functoriality (Gysin maps)}: 
Due to the existence of the fundamental classes introduced in \cite{DJK} the four theories in 
\Cref{df:internal_theories} satisfy exceptional functoriality.\footnote{See \cite[4.3.4]{DJK} 
for the general case of a triangulated motivic category.}
Let $f:Y \rightarrow X$ be a smoothable lci morphism, 
i.e., 
$f$ factors as a regular closed immersion followed by a smooth morphism, 
with cotangent complex $L_f$.
We refer to the associated virtual bundle $\tau_f$ as the virtual tangent bundle of $f$.
One deduces, 
from the system of fundamental classes in \cite[Theorem 3.3.2]{DJK}, 
the canonical natural transformation
\begin{equation}
\label{eq:fdl_class}
\mathfrak p_f:\Th(\tau_f) \otimes f^* \rightarrow f^!
\end{equation}
By adjunction, 
one deduces trace and cotrace maps (see \textsection 4.3.4 in \emph{loc. cit.})
\begin{align*}
\mathrm{tr}_f&:f_!(\Th(\tau_f) \otimes f^*) \rightarrow Id \\
\mathrm{cotr}_f&:Id \rightarrow f_*(\Th(-\tau_f) \otimes f^!)
\end{align*}
These maps induce the following \emph{Gysin maps}:
\begin{itemize}
\item $f^!:\htp_S(X,v) \rightarrow \htp_S(Y,f^{-1}v+\tau_f)$, when $f$ is proper
\item $f_!:\cohtp_S(Y,f^{-1}v+\tau_f) \rightarrow \cohtp_S(X,v)$, when $f$ is proper
\item $f^!:\htp^c_S(X,v) \rightarrow \htp^c_S(Y,f^{-1}v+\tau_f)$
\item $f_!:\cohtp^c_S(Y,f^{-1}v+\tau_f) \rightarrow \cohtp^c_S(X,v)$
\end{itemize}
Again, 
assuming $f$ is proper, 
the comparison maps $\alpha_{X/S}$ and $\alpha'_{X/S}$ are compatible with the above Gysin morphisms in the obvious sense.
\end{num}

The following definition is a variant of \cite[Definition 4.3.7]{DJK}.

\begin{df}
\label{df:hsmooth}
Let $f:X \rightarrow S$ be a smoothable lci morphism with virtual bundle $\tau_f$ over $X$ associated 
to the cotangent complex $L_f$.
We say that $f$ is \emph{homotopically smooth} (\emph{h-smooth}) with respect to the motivic $\infty$-category 
$\T$ if the natural transformation \eqref{eq:fdl_class} evaluated at the sphere spectrum $\un_S$
$$
\mathfrak p_f:\Th(\tau_f) \rightarrow f^!(\un_S)
$$
is an isomorphism.
\end{df}

\begin{num}\label{num:hsmooth}
One gets the following basic properties of h-smoothness:
considering composable lci smoothable morphisms $f$, $g$, $h=f \circ g$ (which is also lci smoothable),
if $f$ and $g$ (resp. $f$ and $h$) are h-smooth, 
then so is $h$ (resp. $g$). 
Moreover, 
if $g^!$ is conservative, $g$ and $h$ being h-smooth implies $f$ is h-smooth. 
On the other hand, h-smoothness is not stable under base change.
\end{num}

\begin{ex}
Here are examples, in order of difficulty, of cases where $f:X \rightarrow S$ is h-smooth:
\begin{itemize}
\item $f$ is smooth
\item $X$, $S$ are smooth over some base $B$ and $f$ is a morphism of $B$-schemes
\item $X$, $S$ are regular over a field $k$ and $\T$ is continuous, see \cite[Appendix A]{DFJK}
(all our examples are continuous in this sense)
\end{itemize}
Note, in particular, that a closed immersion between smooth varieties over a field is h-smooth.
On the other hand, not all regular closed immersions are smooth.
For example, 
the immersion $i:D \rightarrow X$ of a relative normal crossing divisor $D$ over $S$ into a smooth $S$-scheme 
in the sense of Definition \ref{df:normal_crossing} is not h-smooth unless $D$ is smooth over $S$.
Indeed, 
in \Cref{section:Mumford}, 
we show
$$
i^!(\un_X) \simeq \colim_{n \in \Dinj} \bigoplus_{J \subset I, \sharp J=n} i_{J!}(\Th(N_J))
$$
which does not coincide with $\Th(N_DX)$ if there is a non-trivial intersection on the branches of $D$ 
(consider the pullback to any such intersection).
\end{ex}

\begin{rem}
It is expected that any morphism between regular schemes is h-smooth with respect to $\SH(S)$.
The said assertion holds and is called ``absolute purity" in the case when $\T$ is
$\SH(S) \otimes \QQ$ (\cite[Theorem 3.6]{DFJK}),
rational mixed motives $\DM_\QQ(S)$ (\cite[14.4.1]{CD3}),
and \'etale integral motives $\DM_\et(S)$ (\cite[5.6.2]{CD4}).
\end{rem}

\begin{ex}\label{ex:purity&comparison}
The h-smoothness hypothesis allows one to compare the four different theories from Definition \ref{df:internal_theories} and generalize the smooth case. 
More precisely, 
if $f:X \rightarrow S$ is h-smooth, the purity isomorphism induces isomorphisms
\begin{align*}
\htp_S(X,v)=f_!\big(\Th(v) \otimes f^!(\un_S)\big) &\xrightarrow{\mathfrak p_f} f_!\big(\Th(v) \otimes 
\Th(\tau_f)\big)=\cohtp_S^c(X,v+\tau_f) \\
\htp^c_S(X,v)=f_*\big(\Th(v) \otimes f^!(\un_S)\big) &\xrightarrow{\mathfrak p_f} f_*\big(\Th(v) \otimes 
\Th(\tau_f)\big)=\cohtp_S(X,v+\tau_f)
\end{align*}
Moreover, 
these isomorphisms transform the usual functoriality (resp. Gysin maps) in the source to the Gysin maps 
(resp. usual functoriality) on the target. 
This follows by considering the ``associativity formula" for fundamental classes in \cite[Theorem 3.3.2]{DJK}.
\end{ex}

Given this new notion of h-smoothness, 
we extend the Morel-Voevodsky homotopy purity theorem as follows.\footnote{When $Z$ has crossing singularities,  
a more aesthetically pleasing formulation will be given in \Cref{thm:generalizedhomotopypurity2}.}
\begin{thm}
\label{thm:generalizedhomotopypurity}
Let $f:X \rightarrow S$ be an h-smooth morphism with virtual tangent bundle $\tau_f$,
and $v$ a virtual vector bundle over $X$.
Let $i:Z \rightarrow X$ be a closed immersion and set $g=f \circ i$.
We let $\Pi_S(X/X-Z,v)$ denote the homotopy cofiber of the canonical map $j_*:\Pi_S(X-Z,v) \rightarrow \Pi_S(X,v)$.

Then the purity isomorphism $\mathfrak p_f$ induces an isomorphism
$$
\htp_S(X/X-Z,v) \simeq \cohtp_S^c(Z,i^*v+i^*\tau_f)
$$
Moreover, 
if $Z/S$ is h-smooth and $N_i$ denotes the normal bundle associated with the (necessarily regular) closed immersion $i$, 
there exists a (relative) purity isomorphism
$$
\htp_S(X/X-Z,v) \simeq \htp_S(Z,i^*v+N_i)
$$
\end{thm}
\begin{proof}
For the closed immersion $i$ we have the associated localization homotopy exact sequence 
$$
j_!j^! \rightarrow Id \rightarrow i_*i^*
$$
By inserting $\Th(v) \otimes f^!(\un_S)$ and applying $f_!$ we get the homotopy exact sequence
$$
\Pi_S(X-Z,j^*v) \rightarrow \Pi_S(X,v) \rightarrow q_!\big(\Th(i^*v) \otimes i^*f^!(\un_S)\big)
$$
Here, 
we use the identifications
\begin{align*}
f_!j_!j^!\big(\Th(v) \otimes f^!(\un_S)\big) &\simeq 
h_!\big(\Th(j^*v) \otimes j^!f^!(\un_S)\big)=\Pi_S(X-Z,j^*v) \\
f_!i_*i^*\big(\Th(v) \otimes f^!(\un_S)\big) & \simeq 
f_!i_*\big(\Th(i^*v) \otimes i^*f^!(\un_S)\big)=q_!\big(\Th(i^*v) \otimes i^*f^!(\un_S)\big)
\end{align*}
In particular, 
there is an isomorphism
$$
\htp_S(X/X-Z,v) \simeq q_!\big(\Th(i^*v) \otimes i^*f^!(\un_S)\big)
$$
The purity isomorphism yields the desired isomorphism
\begin{align*}
\htp_S(X/X-Z,v) &\simeq q_!\big(\Th(i^*v) \otimes i^*f^!(\un_S)\big) \\
& \xrightarrow{\mathfrak p_f} q_!\big(\Th(i^*v) \otimes i^*\Th(\tau_f) \otimes f^*(\un_S)\big) \\
& = q_!\big(\Th(i^*v+i^*\tau_f)\big)=\cohtp_S^c(Z,i^*v+i^*\tau_f)
\end{align*}
As explained in Example \ref{ex:purity&comparison}, 
the second isomorphism in \Cref{thm:generalizedhomotopypurity} is now a direct consequence of the 
h-smoothness property of $Z/S$.
\end{proof}

\subsection{Homotopy type at infinity and the fundamental class of the diagonal}

\begin{lm}\label{lm:pre-duality}
Let $f:X \rightarrow S$ be a separated morphism and $v$ a virtual vector bundle over $X$.
\begin{enumerate}
\item  There exists a canonical isomorphism
$$
\uHom(\cohtp_S^c(X,v),\un_S) \xrightarrow \simeq \htp_S^c(X,-v)
$$
The isomorphism is functorial in $X$ for both the natural functoriality with respect to proper maps 
(\ref{num:natural_fct}) and the Gysin morphisms with respect to smoothable lci morphisms (\ref{num:Gysin}).
\item Assume, in addition, that $f$ is h-smooth (\ref{df:hsmooth}).
Then there exists an isomorphism
$$
\uHom(\htp_S(X,v),\un_S) \xrightarrow \simeq \cohtp_S(X,-v)
$$
which is again functorial with respect to the natural functoriality and the Gysin maps.
\end{enumerate}
\end{lm}
\begin{proof}
To prove the isomorphism in (1) we use
\begin{align*}
\uHom(\cohtp_S^c(X,v),\un_S)=\uHom(f_!(\Th(v)),\un_S) & \xrightarrow{(a)} f_*\uHom\big(\Th(v),f^!(\un_S)\big) \\
& \stackrel{(b)} \simeq f_*\big(\Th(-v) \otimes f^!(\un_S)\big)=\htp_S^c(X,-v)
\end{align*}
Here, 
(a) (resp. (b)) follows from the internal interpretation of the fact that $f^!$ is right adjoint to $f_!$ 
(resp. $\Th(v)$ is $\otimes$-invertible).
The functoriality statement is obvious by construction.

To deduce (2), we establish the isomorphisms 
\begin{align*}
\uHom(\htp_S(X,v),\un_S)=\uHom\big(f_!\big(\Th(v) \otimes f^!(\un_S)\big),\un_S\big) 
& \xrightarrow{(a)} f_*\uHom\big(\Th(v) \otimes f^!(\un_S),f^!(\un_S)\big) \\
& \stackrel{(b)} \simeq f_*\Big(\Th(-v) \otimes \uHom\big(f^!(\un_S),f^!(\un_S)\big)\Big), \\
& \stackrel{(c)} \simeq f_*\big(\Th(-v) \otimes \un_S\big)=\cohtp_S(X,-v)
\end{align*}
Here, 
(a) and (b) are justified as in (1), 
and (c) follows from the fact that the canonical map
\begin{equation}\label{eq:pre-dualizing}
\un_X \rightarrow \uHom\big(f^!(\un_S),f^!(\un_S)\big)
\end{equation}
obtained by adjunction is an isomorphism because $f$ is h-smooth. 
The functoriality statement follows from the definitions.
\end{proof}

\begin{rem}
We say that $f$ is a \emph{pre-dualizing} morphism if the map \eqref{eq:pre-dualizing} is an isomorphism.
This condition implies the isomorphism in (2).
The notion of a pre-dualizing morphism is crucially linked with Grothendieck-Verdier duality, as shown by \cite[4.4.11]{CD3}. 
In fact, 
if $f^!(\un_S)$ is a dualizing object (\cite[Definition 4.4.4]{CD3}), 
it follows that $f$ is pre-dualizing.
Assuming that $S$ is a smooth $\QQ$-scheme, 
it follows from \cite{Ayoub} that any separated morphism $f:X \rightarrow S$ is pre-dualizing. 
\end{rem}

\begin{thm}\label{thm:diagonal}
Let $f:X \rightarrow S$ be an h-smooth morphism with virtual tangent bundle $\tau_f$.
Then the purity isomorphism $\mathfrak p_f$ induces a canonical isomorphism
$$
\uHom\big(\htp_S(X,v),\un_S\big) \simeq \htp_S^c(X,-v-\tau_f)
$$
In addition, assume that $f$ is smooth.
Then the map $\alpha'_f$ obtained by adjunction from the composite
$$
\htp_S(X) \xrightarrow{\alpha_f} \htp^c_S(X) \simeq \uHom\big(\htp_S(X,-\tau_f),\un_S\big)
$$
fits into the commutative diagram
$$
\xymatrix@=20pt{
\htp_S(X) \otimes \htp_S(X,-\tau_f)\ar^-{\alpha'_f}[r]\ar_\simeq[d] & \un_S \\
\htp_S(X \times_S X,-p_1^*\tau_f)\ar^-{\delta^!}[r] & \htp_S(X)\ar_{f_*}[u]
}
$$
where the left vertical map is the K\"unneth isomorphism.\footnote{Which in our case
is a tautology by definition of the tensor product of spectra and the fact that $X/S$ is smooth.}
\end{thm}
If $f$ is smooth, 
the above commutative diagram shows the map $\alpha'_f$ can be computed under the canonical isomorphism
\begin{align*}
[\htp_S(X),\htp^c_S(X)] \simeq &[\htp_S(X) \otimes \htp_S(X,-\tau_f),\un_S] \\
 &\simeq [X \times_S X,\Th_S(p_1^*)]=:\Pi^{2d,d}(X \times_S X,p_1^*(\tau_f))
\end{align*}
as the fundamental class of the diagonal
$$
\eta_{X \times_S X}(\Delta_{X/S})=\delta_*(1)
$$

\begin{rem}
We note that when $f$ is an h-smooth closed immersion $i:Z \rightarrow X$, 
the diagonal $\delta:Z \rightarrow Z \times_S Z$ is an isomorphism.
In particular, the normal bundle of $\delta$ is unrelated to the virtual tangent space $-N_ZS$ of $Z/S$.
For this reason, the preceding formula cannot be true for arbitrary h-smooth morphisms. 
We expect ideas in derived algebraic geometry can fix this defect.
\end{rem}

\begin{proof}
The first isomorphism is a combination of the second isomorphism in \Cref{lm:pre-duality} 
and the second isomorphism of Example \ref{ex:purity&comparison}
$$
\uHom\big(\htp_S(X,v),\un_S\big) \simeq \cohtp_S(X,-v)
\stackrel{\mathfrak p_f^{-1}} \simeq \htp^c_S(X,-v-\tau_f)
$$
The commutativity of the square follows from the following facts:
\begin{itemize}
\item If $f_1:X \times_S X \rightarrow X$ is the projection on the first factor,
the associativity formula in \cite[Theorem 3.3.2]{DJK} shows there is an equality of fundamental classes 
$\eta_\delta.\eta_{f_1}=1$
\item The assumption that $f$ is smooth implies the cartesian square
$$
\xymatrix@=14pt{
X \times_S X\ar^-{f_1}[r]\ar_{f_2}[d]\ar@{}|\Delta[rd] & X\ar^f[d] \\
X\ar_f[r] & S
}
$$
is Tor-independent. 
Thus the transversal base change formula in \cite[Theorem 3.3.2]{DJK} implies the equality $\Delta^*(\eta_f)=\eta_{f_1}$
\end{itemize}
\end{proof}

\Cref{thm:diagonal} implies the following more potent form of duality 
(see the end of \Cref{sec:notations} for a review).
We give concrete examples in \Cref{thm:strong_duality}.

\begin{cor}[Duality with compact support]\label{cor:Atiyah}
Let $X$ be an h-smooth scheme over $S$ and $v$ a virtual bundle over $X$ such that $\htp_S(X,v)$ is rigid.
Then there exists an isomorphism
$$
\htp_S(X,v)^\vee \simeq \htp_S^c(X,-v-\tau_f)
$$
In particular, 
the Borel-Moore homotopy object $\htp_S^c(X,-v-\tau_f)$ is rigid.
\end{cor}

\begin{ex}
\label{ex:adA1contractible}
Suppose $X$ is a $d$-dimensional smooth affine $\AA^{1}$-contractible $k$-scheme.
Owing to Atiyah duality in \Cref{thm:diagonal} there is an exact homotopy sequence
\begin{equation}
\label{equation:smoothaffinecontractible}
\Pi^{\infty}_{k}(X) \to \Pi_{k}(X) \to \Pi_{k}(X,-T_X)^{\vee}
\end{equation}
The assumption on $X$ implies the tangent bundle $T_X$ is trivial.
Since $\Pi_{k}(X)\simeq \un$ we deduce the equivalences
\[
\Pi_{k}(X,-T_X)^{\vee}
\simeq
(\Pi_{k}(X)(-d)[-2d])^{\vee}
\simeq
\Pi_{k}(X)(d)[2d]
\simeq 
\un(d)[2d]
\]
Morel's connectivity theorem shows the map $\un \to \un(d)[2d]$ is zero for $d>0$.
Hence there is an equivalence 
\[
\Pi^{\infty}_{k}(X)
\simeq
\un\oplus \un(d)[2d-1]
\]
In conclusion, 
the stable motivic homotopy type at infinity cannot distinguish between smooth affine $\AA^{1}$-contractible 
$k$-schemes of the same dimension.
This is analogous to the classification of Euclidean spaces among open contractible manifolds, 
see \cite{asokostvar},
which cannot be achieved by invariants of the stable topological homotopy category.
\end{ex}

\subsection{Stable motivic homotopy type at infinity of hyperplane arrangements}
\label{subsection:smhtaioha}

\begin{num}
In the following, 
we use \Cref{thm:diagonal} to compute the stable motivic homotopy type at infinity of hyperplane arrangements.
We say a smooth $S$-scheme $X$ with structural morphism $f$ is \emph{stably $\AA^1$-contractible} if 
$f_*:\Sigma^\infty X_+ \rightarrow \un_S$ is an equivalence.
Vector bundles give basic examples over $S$.
To illustrate the preceding results, we determine the stable homotopy of a normal crossing $S$-scheme.
\end{num}

\begin{prop}
\label{prop:complementdivisor}
Let $S$ be a stably $\AA^1$-contractible smooth scheme over a field $k$,
and $X$ be a smooth affine and stably $\AA^1$-contractible $S$-scheme of dimension $d$.
Suppose the closed subscheme $D=\cup_{i \in I} D_i$ of $X$ is a smooth reduced crossing $S$-scheme in the sense of 
\Cref{df:normal_crossing}.
In addition, assume that for any $J \subset I$, 
every connected component of $D_J$ is stably $\AA^1$-contractible over $S$.
For a subset $J \subset I$ we set $n_J=\sharp J$, 
and for any generic point $x$ of $D_J$ we let $c_J(x)$ denote the codimension of $x$ in $X$.

Then there exists a canonical isomorphism
$$
\Pi_S(X-D) \simeq \bigoplus_{J \subset I, x \in D_J^{(0)}} \un_S\big(c_J(x)\big)\big[2c_J(x)-n_J\big]
$$
If $D$ is a normal crossing divisor
and $m(n)$ is the sum of the number of connected components of all codimension $n$ subschemes $D_J$ of $X$, 
then the isomorphism takes the form
$$
\Pi_S(X-D) \simeq \bigoplus_{n=0}^d m(n).\un_S(n)[n]
$$
\end{prop}
\begin{proof}
According to Corollary \ref{cor:snccorollary}, one obtains that $\Pi_S(X-D)$ is the homotopy limit of the (finite) tower
\begin{equation}
\label{equation:finitetower}
\Pi_S(X) \rightarrow  \oplus_{i \in I} \Pi_S(D_i,N_i)
\rightarrow \cdots \rightarrow \oplus_{J \subset I, \sharp J=n} \Pi_S(D_J,N_J) \rightarrow \cdots
\end{equation}
Suppose $x$ is a generic point on $D_J\neq \varnothing$, for $J \subset I$, 
and write $D_J(x)$ for the associated connected component.
By assumption, $D_J(x)$ is smooth and stably $\AA^1$-contractible over $k$.
Thus there is an isomorphism
$$
\ZZ=
K_0(k)=
\Hom_{\SH(k)}(\un_k,\KGL) 
\xrightarrow 
\cong 
\Hom_{\SH(k)}(\Sigma^\infty D_{J+},\KGL) 
=
K_0(D_J)
$$
In particular, 
the rank $c_J(x)$ vector bundle $N_J|_{D_J(x)}$ is stably trivial. 
Since $D_J(x)/S$ is stably $\AA^1$-contractible,
it follows that
$$
\Pi_S(D_J,N_J) 
\simeq 
\oplus_x \Pi_S(D_J(x),N_J|_{D_J(x)}) 
\simeq 
\oplus_x \un_S(c_J(x))[2c_J(x)]
$$
To deduce the first assertion it suffices to show that the morphisms in \eqref{equation:finitetower} are zero.
Recall that these maps are sums of Gysin morphisms $\nu_K^{J!}$ for $J,K \subset I$, $K=J \cup\{k\}$,
$\nu_K^J:D_K \rightarrow D_J$. 
We are reduced to consider maps of the form
\begin{equation}
\label{equation:finitemap}
\un_S(c_J(x))[2c_J(x)] \rightarrow \un_S(c_K(y))[2c_J(y)]
\end{equation}
Here, 
$x$ (resp. $y$) is a generic point of $D_J$ (resp. $D_K$). 
Since $D_J$ is nowhere dense in $D_K$,
all such maps belong to some stable cohomotopy group $\pi^{2r,r}(S)$ for $r>0$.
The assumption that $S$ is stably $\AA^1$-contractible over $k$ implies $\pi^{2r,r}(S) \simeq \pi^{2r,r}(k)$.
Morel's $\AA^1$-connectivity theorem shows the latter group is trivial. 
It follows that the map \eqref{equation:finitemap} is zero.

For the second assertion, 
it suffices to note that if $D$ is a normal crossing divisor,
then $c_J=n_J$ for any $J \subset I$.
\end{proof}


\begin{ex} 
Let $S/k$ be a stably $\AA^1$-contractible smooth scheme.
\Cref{prop:complementdivisor} applies to the basic example when $X \simeq \AA^d_S$ and $D$ is a hyperplane arrangement in $X$.
\end{ex}

Owing to \Cref{cor:Atiyah} and \Cref{prop:complementdivisor}, 
we deduce the following computation of stable motivic homotopy types at infinity.

\begin{prop}
\label{prop:complementdivisor2}
Under the assumptions in \Cref{prop:complementdivisor} the object $\Pi_S(X-D)$ is rigid (see \Cref{sec:notations})
and there are isomorphisms
$$
\Pi_S^c(X-D) 
\simeq 
\Pi_S(X-D)^\vee(d)[2d] 
\simeq 
\bigoplus_{K \mid D_K \neq \varnothing} \un_S(d-c_K)[2(d-c_K)+n_K]
$$
Moreover, there exists a canonical isomorphism
$$
\Pi^\infty_S(X-D)\simeq 
\bigoplus_{J \mid D_J \neq \varnothing} \un_S(c_J)[2c_J-n_J]
\oplus \bigoplus_{K \mid D_K \neq \varnothing} \un_S(d-c_K)[2(d-c_K)+n_K-1]
$$
Finally, 
if $D$ is a normal crossing divisor in $X$ relative to $S$,
then we have 
$$
\Pi^\infty_S(X-D) \simeq 
\bigoplus_{n=0}^d m(n).\un_S(n)[n] \oplus \bigoplus_{n=0}^d m(n).\un_S(d-n)[2d-n-1]
$$
\end{prop}

\section{Normal crossing boundaries}
\label{section:NCboundary}

\subsection{Homotopy type of crossing singularities}
\label{sec:crossing_sing}

We start with various notion of singularities extending the classical notion of normal crossing divisor.
\begin{df}
\label{df:normal_crossing}
Let $X$ is an $S$-scheme with irreducible components $X_\bullet=(X_i)_{i \in I}$.
For $J \subset I$, 
let $X_J=\cap_{j \in J} X_j$, where $\cap$ means fiber product over $X$ and $X'_J=(X_J)_{red}$.
We say that $X$ has \emph{smooth} (resp. \emph{regular}, \emph{h-smooth}) \emph{crossing} over $S$ if, 
for any non-empty $J \subset I$, 
$X'_J$ is a smooth (resp. regular, h-smooth, \Cref{df:hsmooth}) $S$-scheme.

Suppose, 
in addition, 
that $X$ is a sub-scheme of a smooth (resp. regular, h-smooth) $S$-scheme $\Omega$.
Then $X$ is called a normal (resp. regular, h-normal) crossing subscheme of $\Omega$ if $X$ is a smooth 
(resp. regular, h-smooth) crossing $S$-scheme, and for all $J \subset I$ there is the equality
$$
\codim_\Omega(X_J)=\sum_{j \in J} \codim_\Omega(X_j)
$$
\end{df}

\begin{rem}~
\begin{enumerate}
\item When $X$ is a divisor of $\Omega$, 
i.e., 
every irreducible component of $X$ has codimension $1$ in $\Omega$,
we recover the classical notion of a normal crossing divisor relative to the base $S$.
\item Working with the reduced scheme $X'_J$ instead of $X_J$ allows to consider multiplicities on the 
intersections of the irreducible components of $X$.
That is, 
we don't restrict to simple normal crossing divisors.
\end{enumerate}
\end{rem}


\begin{num}\label{num:ordeded_Cech}
Let $X$ be a noetherian scheme and consider a finite closed cover\footnote{In applications, 
it is the set of irreducible components,
but the flexibility with closed covers simplifies the proofs.} of $X$, 
i.e., 
a surjective map 
$$
p:\sqcup_{i \in I} X_i \rightarrow X
$$
obtained from closed immersions $X_i \rightarrow X$.
For the sake of intuitive notation,
let $\cap=\times_X$ be shorthand for the fiber product of closed $X$-schemes.
The \v Cech simplicial $X$-scheme $\cech_*(X_\bullet/X)$ associated with $p$ takes the form
\begin{equation}
\label{equation:cechsimplicial}
\cech_n(X_\bullet/X)=\bigsqcup_{(j_0,\hdots,j_n) \in I^{n+1}} X_{j_0} \cap \hdots \cap X_{j_n}
\end{equation}
The degeneracy maps $\delta_n^i:\cech_n(X_*) \rightarrow \cech_{n-1}(X_*)$ are given by the appropriate closed immersions.
One can simplify $\cech_*(X_\bullet/X)$ by setting $X_J=\cap_{j \in J} X_j$ for each subset $J \subset I$.
To proceed we consider semi-simplicial schemes.\footnote{Let $\Dinj$ be the category of finite ordered sets with morphisms 
only the injective maps.
A semi-simplicial object in a category $\mathscr C$ is a contravariant functor from $\Dinj$ to $\mathscr C$. 
There is a forgetful functor from simplicial objects to semi-simplicial objects that forgets the degeneracy maps.}
Let us choose a total order on the finite set $I$, 
and define the ordered \v Cech semi-simplicial $X$-scheme $\cecho_*(X_\bullet)$ associated to $X_\bullet/X$ by setting
\begin{equation}
\label{equation:cechsemisimplicial}
\cecho_n(X_\bullet/X)
:=
\bigsqcup_{J \subset I, \sharp J=n+1} X_J
\end{equation}
The degeneracy maps in \eqref{equation:cechsemisimplicial} admit the following description.
To any subset $J \subset I$ of cardinality $n+1$,
we associate the uniquely defined $n$-tuple $(j_0,\hdots,j_n) \in I^{n+1}$ such that $J=\{j_0,\hdots,j_n\}$ and $j_0<\hdots<j_n$.
This gives a canonical embedding $\cecho_*(X_\bullet/X) \subset \cech_*(X_\bullet/X)$ of $\NN$-graded $X$-schemes.
In degree $n$, 
the factor $X_J$ maps by the identity to $X_{j_0} \times_X \hdots \times_X X_{j_n}$.
The degeneracy maps $\delta_n^i$ in the simplicial structure on $\cech_*(X_\bullet/X)$ preserve $\cecho_*(X_\bullet/X)$.
This yields the desired semi-simplicial structure on $\cecho_*(X_\bullet/X)$.
Moreover, 
the map $p$ induces an augmentation $\cecho_*(X_\bullet/X)\to X$.
\end{num}

\begin{rem}
The ordered \v Cech semi-simplicial scheme is much smaller than $\cech_*(X_\bullet/X)$.
It is bounded by the cardinality $c$ of the index set $I$ in the sense that $\cecho_n(X_\bullet/X)=\varnothing$ for all $n>c$ 
(a sum indexed by the empty set is the empty scheme).
\end{rem}

\begin{num}
Owing to \cite[Appendix A]{BRTV} the stable homotopy category $\SH(S)$ has an underlying stable $\infty$-category $\iSH(S)$, 
and the six functor formalism is compatible with the $\infty$-structure. 
We shall use a particular case of the said functoriality. 
Let $\Sch_S$ (resp. $\Sche_S$) be the category of separable schemes (resp. schemes separated of finite type) over $S$. 
To any motivic spectrum $\E$ over $S$,
we associate the contravariant $\infty$-functor
$$
\Sch_S^{op} \xrightarrow{\Pi'_S(-,\E)} \iSH(S), (f:X\rightarrow S) \mapsto f_*f^*(\E)
$$
This is the $\infty$-categorical version of the notion of relative cohomology with coefficients in $\E$.
Dually, 
we obtain the covariant $\infty$-functor
$$
\Sche_S \xrightarrow{\Pi_{S}(-,\E)} \iSH(S), (f:X\rightarrow S) \mapsto f_!f^!(\E)
$$
\end{num}

\begin{num}
Let us return to the setup in \Cref{num:ordeded_Cech}, 
with $X=\sqcup X_{i \in I}$. 
Suppose $p\colon X \rightarrow S$ is a morphism of finite type.
For any $J \subset I$,
we let $p_J\colon X_J \rightarrow S$ denote the obvious composition.
To the ordered \v Cech complex $\cecho_n(X_\bullet/X)$ and $\E$ we associate the functors
\begin{align*}
(\Dinj)^{op} &\rightarrow \Sche_S\xrightarrow{\Pi_{S}(-,\E)} \iSH(S) \\
\Dinj &\rightarrow \Sch_S^{op} \xrightarrow{\Pi'_S(-,\E)} \iSH(S)
\end{align*}
By using the augmentation map to $X$, 
we obtain canonical maps involving the limit and colimit of the preceding functors
\begin{align*}
\pi_{X_\bullet/X,\E}:&\colim_{n \in (\Dinj)^{op}} \left(\oplus_{J \subset I, \sharp J=n+1} \Pi_{S}(X_J,\E)\right) 
\rightarrow \Pi_{S}(X,\E) \\
\pi'_{X_\bullet/X,\E}:&\Pi'_S(X,\E) \rightarrow 
\underset{n \in \Dinj}\lim\left(\oplus_{J \subset I, \sharp J=n+1} \Pi'_S(X_J,\E)\right)
\end{align*}

We show the colimit (resp. limit) can be viewed as the ``standard'' resolution of homology (resp. cohomology) of 
$X/S$ with coefficients in $\E$.
\end{num}
\begin{thm}
\label{thm:compute_closed_cover}
Under the above assumptions, 
$\pi_{X_\bullet/X,\E}$ and $\pi'_{X_\bullet/X,\E}$ are isomorphisms in $\SH(S)$.
\end{thm}
\begin{proof}
First we reduce to the case when $X$, and similarly each $X_i$, is reduced.
We consider the nil-immersion $\nu:X_{red} \rightarrow X$. 
The localization property of $\SH$ implies the adjunction $(\nu^*,\nu_*)$ is an equivalence.
Moreover, 
$\nu^*$ is right adjoint to $\nu_*=\nu_!$, 
so that $(\nu_!,\nu^!)$ is an equivalence of categories. 
This implies $\nu_*:\Pi_{S}(X_{red},\E) \rightarrow \Pi_{S}(X,\E)$ is an isomorphism.
\vspace{0.05in}

Let us consider $\pi_{X/S,\E}$.
For any $J \subset I$,
there is an isomorphism $p_{J!}p_J^! \simeq p_!i_{J!}i_J^!p^!$,
where $i_J:X_J \rightarrow X$ is the canonical closed immersion.
By replacing $\E$ with $p^!(\E)$, 
we have reduced to the case $X=S$. 
There is, see for example \cite[B.20]{DFJK}, a conservative family of functors
$$
i_x^!
\colon 
\iSH(X) \rightarrow \iSH\big(\Spec(\kappa(x))\big), x \in X
$$
Therefore, 
it suffices to show $i_x^!\big(\pi_{X_\bullet/X,\E}\big)$ is an isomorphism for all $x \in X$.
Given $J \subset I$, 
we consider the following cartesian square
$$
\xymatrix@=24pt{
X'_J\ar_{i'_J}[d]\ar^{i'_x}[r] & X_J\ar^{i_J}[d] \\
\{x\}\ar^{i_x}[r] & X
}
$$
Proper base change for $i_J$ implies there is an isomorphism 
$$
i_x^!i_{J!}i_J^! \simeq i'_{J!}i_x^{\prime!}I_J^!
$$
Moreover, 
we have 
$$
i'_{J!}i_x^{\prime!}I_J^! \simeq i'_{J!}i_J^{\prime !}i_x^!
$$
For the map in question, 
we deduce
$$
i_x^!\big(\pi_{X_\bullet/X,\E}\big)
=
\pi_{X_\bullet \times_X \{x\}/\{x\},i_x^!\E}
$$
We may therefore assume $X=\{x\}$ is a point.  
In this case, 
since the $X_i$'s are closed reduced subschemes of the reduced scheme $\{x\}$, 
the closed cover $p:\sqcup_{i \in I} X_i \rightarrow X=\{x\}$ is given by a sum of identity maps.  
To proceed, 
one can, 
for example, 
construct an explicit homotopy contraction of the semi-simplicial augmented pointed $X$-scheme
$$
\cecho_*(X_\bullet/\{x\})_+ \rightarrow \{x\}_+
$$

The proof for the map $\pi'_{X_\bullet/X,\E}$ is entirely analogous.
In this case, we employ, 
see for example \cite[Proposition 4.3.17]{CD3}, 
the conservative family of functors
$$
i_x^*:\iSH(X) \rightarrow \iSH\big(\Spec(\kappa(x))\big), x \in X
$$
\end{proof}

\begin{rem}
The proof method of \Cref{thm:compute_closed_cover} can be used to show the isomorphisms 
\begin{equation}
\label{equation:firstcomputation}
\colim_{n \in (\Dinj)^{op}} \left(\oplus_{i_0 \leq \hdots \leq i_n} \Pi_{S}(X_{i_0,\hdots,i_n},\E)\right)
\simeq 
\Pi_{S}(X,\E)
\end{equation}
\begin{equation}
\label{equation:secondcomputation}
\colim_{n \in (\Delta)^{op}} \left(\oplus_{(i_0,\hdots,i_n) \in I^n} \Pi_{S}(X_{i_0,\hdots,i_n},\E)\right)
\simeq 
\Pi_{S}(X,\E)
\end{equation}
The isomorphism \eqref{equation:secondcomputation} is a manifestation of $\cdh$-descent for the 
\emph{triangulated motivic category} $\SH$, 
see \cite[Proposition 3.3.10]{CD3}.
In \eqref{equation:secondcomputation} we may replace $\Delta$ by $\Dinj$ since the inclusion functor 
$\Dinj \subset \Delta$ is final.
We leave the analogous dual statements for $\Pi'_S(X,\E)$ to the reader.
\end{rem}

\begin{rem}
\label{rem:A1_hlg&motives}
\Cref{thm:compute_closed_cover} holds more generally for any triangulated motivic category $\mathscr T$ 
in the sense of \cite[2.4.45]{CD3}
which admits an $\infty$-categorical enhancement as in \cite{Khan}.
As shown in \emph{loc. cit.}, 
such an enhancement exists if $\mathscr T$ is defined via a model $\Sm$-premotivic structure, 
as in \cite[\textsection 1.3.d]{CD3}.
Main examples include the motivic stable homotopy category with coefficients (see \cite[\textsection 4.5]{Ayoub2}),
the $\AA^1$-derived category (see \cite[5.3.31]{CD3}) rational (or Beilinson) motives $\DM_\QQ$ (see \cite[Part IV]{CD3}),
étale motives $\DM_\et$ (see \cite{CD4}), 
$\cdh$-motives (see \cite{CD5}), 
and \'etale torsion and $\ell$-adic sheaves 
(see \cite{LiuZheng} respectively \cite[\textsection 7.2]{CD4}).
\end{rem}

\begin{ex}
\label{ex:sncexample1}
Suppose $X$ has smooth reduced crossings over $S$, and let $X_\bullet$ be the closed cover by the irreducible components of $X$.
Set $c=\sharp I$ following \Cref{df:normal_crossing}.
In this case, 
the isomorphism $\pi_{X_\bullet/X,\un_S}$ from \Cref{thm:compute_closed_cover} amounts to the 
exact homotopy sequence
\begin{equation}
\label{equation:snches}
\Sigma^\infty X'_{I+} 
\rightarrow 
\hdots
\bigoplus_{J \subset I, \sharp J=2} \Sigma^\infty X'_{J+}
\rightrightarrows \bigoplus_{i \in I} \Sigma^\infty X'_{i+}
\rightarrow \htp_S(X)
\end{equation}
of length at most $c+1$ in $\SH(S)$ (note that $X'_I$ can be empty).
For $J \subset K$ there is a naturally induced inclusion $\nu_K^J:X'_K \rightarrow X'_J$.
The face map 
$$
\delta_n^k
\colon
\bigoplus_{K \subset I, \sharp K=n+1} \Sigma^\infty X'_{K+} \rightarrow \bigoplus_{J \subset I, \sharp J=n} \Sigma^\infty X'_{J+}
$$
in \eqref{equation:snches} is defined by the formula
$$
\delta_n^k
=
\sum_{J=\{i_0<\hdots<\widehat{i_k}<\hdots<i_n\} \subset K=\{i_0<\hdots<i_n\}  } (\nu_{K}^J)_*
$$
We obtain a similar formula when considering a twist by a virtual bundle $v$ over $X$
\begin{equation}
\label{equation:snches2}
\Sigma^\infty \Th_S(v_I)
\rightarrow 
\hdots
\bigoplus_{J \subset I, \sharp J=2} \Sigma^\infty \Th_S(v_J)
\rightrightarrows \bigoplus_{i \in I} \Sigma^\infty \Th_S(v_i)
\rightarrow \htp_S(X,v)
\end{equation}
where $v_J$ is the pullback of $v$ to $X'_J$ for $J \subset I$.
\end{ex}

\begin{rem}
\label{rem:compare}
Continuing with \Cref{ex:sncexample1}, 
the $S$-scheme $X$ defines a sheaf of sets $\underline X$ on $\Sm_{S}$. 
We claim the preceding computation yields an isomorphism $\htp_S(X)\simeq \Sigma^{\infty}\underline X_{+}$.
A proof uses the $\PP^1$-stable $\AA^1$-homotopy category $\underline \SH_\cdh(S)$ over $S$ for the big cdh 
site\footnote{The site of finite type $S$-scheme endowed with the $\cdh$-topology.}
in the style of \cite[\textsection 6.1]{CD3}.\footnote{In the terminology of \emph{loc. cit.},
one gets an \emph{enlargement} of motivic triangulated categories: see Definition 1.4.13.}
Theorem \ref{thm:compute_closed_cover} holds in $\underline \SH_\cdh(S)$ due to $\cdh$-descent, 
and then the comparison reduces to the smooth case, 
which holds by the general properties of an enlargement.
\end{rem}

\begin{ex}
\label{ex:sncexample2}
With the same assumptions as in \Cref{ex:sncexample1} and following Remark \ref{rem:A1_hlg&motives}, 
the \'etale or rational motive $M_S(X)=f_!f^!(\un_S)$ of $X/S$ is isomorphic to the complex of sheaves
$$ 
M_S(X'_{I}) 
\xrightarrow{d_{c-2}} 
\hdots 
\xrightarrow{\ d_1\ }
\bigoplus_{J \subset I, \sharp J=2} M_S(X'_{J})
\xrightarrow{\ d_0\ } \bigoplus_{i \in I} M_S(X'_{i})
$$
The differentials are given by the alternating sums $d_i=\sum_{k=0}^n (-1)^k\delta_n^k$.
Equivalently, 
$M_S(X)$ is the motivic complex associated with the \'etale sheaf (Nisnevich sheaf for rational motives) 
represented by $X$ on $\Sm_S$, 
according to Remark \ref{rem:A1_hlg&motives}.

This formula is a motivic relative version of the classical computation of the homology of a normal crossing scheme.
It actually gives back the known formulas by realization of motives (Betti, \'etale, etc...). 
A dual formula, 
see also \Cref{thm:strong_duality},
holds for computing the relative Chow motive $h_S(\underline X)=p_*p^*(\un_S)$, $p:X \rightarrow S$, 
by considering the isomorphism $\pi'(X_\bullet/X,\un_S)$ of Theorem \ref{thm:compute_closed_cover}: 
$h_S(\underline X)$ is quasi-isomorphic to the image of the preceding complex under the (derived) internal Hom functor 
$\derR \uHom(-,\un_S)$.
\end{ex}

As a consequence of \Cref{thm:compute_closed_cover}, 
we obtain the following generalization of a computation due to Rappoport and Zink (see \Cref{rem:RZS} for details).

\begin{cor}
\label{cor:snccorollary}
Let $i:Z \rightarrow X$ be a closed immersion, $U=X-Z$ and $j:U \rightarrow X$ the complementary open immersion.
For the decomposition into irreducible components $Z=\cup_{i \in I} Z_i$ we set $c=\sharp I$, 
$Z_J=\cap_{j \in J} Z_j$, 
where $\cap$ is the fiber product over $X$, 
and $Z'_J=Z_{J,red}$. 
Assume one of the following conditions holds:
\begin{enumerate}
\item $\mathscr T=\SH_\QQ$, $\DM_\QQ$, $\DM_\et$, $D^b_{c}(-_\et,\ZZ_\ell)$, $\mathrm{DH}(X)$,
$X$ is regular and $D$ is a regular crossing subscheme of $X$
\item $\mathscr T$ is any of the triangulated motivic categories in \Cref{sec:notations},
$X$ is h-smooth over a scheme $S$ (see \Cref{df:hsmooth}) and $Z$ is an h-smooth crossing subscheme of $X/S$
\end{enumerate}

Then the object $i^*j_*(\un_U)$ of $\mathscr T(Z)$ is isomorphic to the colimit of the following augmented 
semi-simplicial object of length at most $c+1$ in the underlying $\infty$-category
$$
\nu_{I*} \Th(-N_I) \rightarrow \hdots \bigoplus_{J \subset I, \sharp J=2} \nu_{J*}(\Th(-N_J)) 
\rightrightarrows
\bigoplus_{i \in I} \nu_{i*}(\Th-(N_i)) \xrightarrow \epsilon \un_Z
$$
Here, 
$N_J$ is the normal bundle of the regular closed immersion $Z'_J \rightarrow X$, and the face maps are defined 
by the Gysin maps
$$
\delta_n^k=
\sum_{J=\{i_0<\hdots<\widehat{i_k}<\hdots<i_n\} \subset K=\{i_0<\hdots<i_n\}  }(\nu_{K}^J)_!
$$
Here, 
$(\nu_K^J)_!\colon \nu_{K*}(\Th(N_K)) \rightarrow \nu_{J*}(\Th(N_J))$ is the Gysin map associated with the 
closed immersion $\nu_J^K$ relative to $X$.\footnote{In fact, 
one can interpret $\nu_{J*}(\Th(N_J))$ as the twisted cohomology $\Pi'_X(Z_J,\twist{N_J})$,
so that $(\nu_J^K)_!$ is the Gysin map associated with the proper morphism of $X$-schemes $\nu_J^K$.}
Similarly, 
the Gysin maps $\nu_{i!}:\nu_{i*}(\Th(N_i)) \rightarrow \un_Z$ for all $\nu_i:Z_i \rightarrow Z$, $i\in I$, 
define the augmentation map
$$
\epsilon=\sum_{i \in I} \nu_{i!}
$$

Dually,
in the underlying $\infty$-category, 
the object $i^!j_!(\un_U)$ in $\mathscr T(U)$ is isomorphic to the limit of the following augmented 
semi-cosimplicial object (of length at most $c+1$)
$$
\un_Z 
\rightarrow 
\bigoplus_{i \in I}\nu_{i!}(\Th(N_i))
\rightrightarrows \bigoplus_{J \subset I, \sharp J=2} \nu_{J!}(\Th(N_J)) \hdots
\rightarrow 
\nu_{I!} \Th(N_I)
$$
\end{cor}
\begin{proof}
We note that $Z'_J \rightarrow X$ is a regular closed immersion so that the normal bundle $N_J$ exists.
The distinguished localization triangle associated with the closed immersion $i$ corresponds to an 
exact homotopy sequence in the underlying $\infty$-category
$$
i_!i ^!(\un_S) \rightarrow \un_S \rightarrow j_*j^*(\un_S)=j_*(\un_U)
$$
Applying $i^*$ and \Cref{thm:compute_closed_cover} to the left-hand term, 
we deduce the exact homotopy sequence
$$
\colim_{n \in (\Dinj)^{op}} \big(\nu_{J*}i_J^!(\un_S)\big) \rightarrow \un_Z \rightarrow i^*j_*(\un_U)
$$
To conclude, 
we use the purity isomorphism $i_J^!(\un_S)\simeq \Th(-N_J)$, 
which is absolute purity in case of assumption (1) and the fact $i_J$ is h-smooth in case of assumption (2)
(apply \Cref{num:hsmooth}).
\vspace{0.1in}

The result for $i^!j_!(\un_U)$ follows by applying the same argument to the dual exact homotopy sequence
$$
j_!j ^!(\un_S) \rightarrow \un_S \rightarrow i_*i^*(\un_S)
$$ 
\end{proof}

\begin{rem}
\label{rem:RZS}
In the situation of \Cref{cor:snccorollary} we consider the cases, 
(1) $\mathscr T=\DM_\QQ, \DM_\et$, 
(2) $\mathscr T=\DM$, 
and assume $Z=D$ is a normal crossing divisor on $X$.
Then, 
in the underlying motivic $\infty$-category, 
the motive $i^*j_*(\un_U)$ is the colimit of the complex 
\begin{equation}
\label{equation:RZgeneralization}
\nu_{I*}(\un)\twist{c} \xrightarrow{d_{c-2}} \hdots \bigoplus_{J \subset I, \sharp J=2} \nu_{J*}(\un)\twist{2} \xrightarrow{d_0}
\bigoplus_{i \in I} \nu_{i*}(\un)\twist{1} \xrightarrow \epsilon \un_D
\end{equation}
Here, 
$d_n=\sum_k (-1)^k \delta_n^k$ is an alternate sum of Gysin maps associated with the relevant closed immersions.
The computation for \eqref{equation:RZgeneralization} specializes under $\ell$-adic realization to the Rapoport-Zink 
formula for vanishing cycles \cite[Lemma 2.5]{RapZink}.
A similar remark applies to Steenbrink's limit Hodge structure \cite{Steenb},
except that our computation for motives does not account for the action of the monodromy operator.
\end{rem}

\begin{cor}\label{cor:snccorollary2}
Consider the assumptions of \Cref{cor:snccorollary},
 and assume in addition (in case (1)) that $X$ be an $S$-scheme with projection map $\bar p:X \rightarrow S$.

Then $\Pi_{S}(U)$ is isomorphic to the limit of the augmented semi-simplicial diagram
\begin{equation}
\label{equation:poles}
\Pi_{S}(X) 
\xrightarrow \epsilon \bigoplus_{i \in I} \Pi_{S}(Z_i,\twist{N_i}) 
\rightrightarrows \bigoplus_{J \subset I, \sharp J=2} \Pi_{S}(Z_J,\twist{N_J}) \hdots 
\rightarrow \Pi_{S}(Z_I,\twist{N_I})
\end{equation}
given by the maps 
\begin{align*}
\epsilon&=\sum_{i \in I} \nu_i^! \\
\delta^n_k&=\sum_{K=\{i_0<\hdots<i_n\}, J=\{i_0<\hdots<\not{i_k}<\hdots<i_n\}} (\nu_{K}^J)^!
\end{align*}

Similarly, 
the motive $M_S(U)$ is the limit of the augmented semi-simplical diagram
\begin{equation}
\label{eq:poles}
M_S(X) \xrightarrow \epsilon \bigoplus_{i \in I} M_S(Z_i)\twist{1}
\rightrightarrows \bigoplus_{J \subset I, \sharp J=2} M_S(Z_J)\twist{2} \hdots 
\rightarrow M_S(Z_I)\twist{c}
\end{equation}
with the same formulas as in \eqref{equation:poles} for the augmentation $\epsilon$ and the coface maps $\delta^n_k$.
\end{cor}
\begin{proof}
Similarly to \Cref{cor:snccorollary},
inserting $p^!(\un_S)$ in the localization distinguished triangle for the closed immersion $i$ and applying $\bar p_!$ 
yields the exact homotopy sequence
$$
M_S(U)=\bar p_!j_!j ^!\bar p^!(\un_S) \rightarrow M_S(X)=\bar p_!\bar p^!(\un_S) \rightarrow \bar p_!i_*i^*\bar p^!(\un_S)
$$
\Cref{thm:compute_closed_cover} identifies the right-hand term with 
$$
\lim_{ n \in \Dinj} \sum_{J \subset I, \sharp J=n+1} \bar p_!i_{J!}i_J^*\bar p^!(\un_S)
$$
Under the assumptions (1), (2) we have $\bar p^!(\un_S)\simeq \Th(T_{X/S})$.
Owing to purity, absolute or relative, for the regular closed immersion $i_J$ we deduce the isomorphism
$$
\bar p_!i_{J!}i_J^*\bar p^!(\un_S) \simeq \bar p_!i_{J!}\big(\Th(N_J) \otimes i_J^!\bar p^!(\un_S)\big)=\Pi_{S}(Z_J,\twist{N_J})
$$
The computation for $M_S(U)$ follows because motives are oriented.
\end{proof}

\begin{rem}
Suppose $X/S$ is smooth and proper.
The formula for $M_S(X-D)$, 
the motive of the complement of a normal crossing divisor $D$ of $X/S$, 
is a relative motivic analog of the de Rham complex with logarithmic poles that Deligne used to define mixed Hodge structures. 
The motive of the relative non-compact family $X-D$ over $S$ is expressed as the ``complex" \eqref{eq:poles} whose terms are 
pure of weight $0$ for Bondarko's motivic weight structure. 
In particular, 
it gives a canonical and functorial weight filtration for the motive $M_S(X-D)$ (recall that a pure object of weight $0$ 
shifted $n$ times has weight $n$).
This is a motivic analog of the fact that the weight filtration of the mixed Hodge structure on $X-D$, 
at least when $S=\Spec(\CC)$, 
is obtained by the naive filtration of the de Rham complex with logarithmic poles associated with $(X, D)$.
\vspace{0.1in} 

Dually, 
we can identify the Chow motive $h_S(X)=p_*(\un_S)$ with the colimit of the diagram
\begin{equation}
\label{equation:chowmotive}
h_S(D_I)\twist{-c} \rightarrow \hdots \bigoplus_{J \subset I, \sharp J=2} h_S(D_J)\twist{-2}
\rightrightarrows \bigoplus_{i \in I} h_S(D_i)\twist{-1} 
\xrightarrow 
\epsilon h_S(X)
\end{equation}
When $S=\Spec(\CC)$, 
the de Rham realization of \eqref{equation:chowmotive}, 
see \cite[\textsection 3.1]{CD2}, 
can be canonically identified with the de Rham complex with logarithmic poles associated with
$(X, D)$.\footnote{As explained in \cite[Example 5.4.2(1)]{Deg12}, 
this follows from the identification of the orientation of the motivic spectrum representing algebraic de Rham cohomology.}
\end{rem}

\subsection{Strong duality}

We deduce some applications of the computations in the previous section towards strong duality results.
\begin{thm}\label{thm:strong_duality}
\begin{enumerate}
\item Let $X$ be a proper scheme over $S$ with smooth reduced crossings,
and $v$ be a virtual bundle over $X$.
Then $\htp_S(X,v)$ is rigid with dual $\cohtp_S(X,-v)$.
\item Let $X/S$ be scheme which admits a smooth proper compactification $\bar X$ whose complement $\partial X$ 
is a smooth reduced crossing $S$-scheme.
Then $\htp_S(X,v)$ is rigid with dual $\htp^c_S(X,-v-T_X)$, 
where $T_X$ is the tangent bundle of $X/S$.
\end{enumerate}
\end{thm}
\begin{proof}
For point (1): 
according to \Cref{equation:snches2}, $\htp_S(X,v)$ is a colimit of a finite diagram whose components are rigid spectra
(as they are given by spectra of smooth proper schemes). 
This implies $\htp_S(X,v)$ is rigid. 
The fact that its dual is $\cohtp_S(X,-v)$ follows from \Cref{lm:pre-duality}(1).

For point (2), one applies \Cref{cor:snccorollary2} and the first assertion of \Cref{thm:diagonal}.
\end{proof}

\begin{rem}
\label{rem:strong_duality}
In case (1), the rigid spectrum $\htp_S(X,v)$ is isomorphic to the homotopy colimit of the diagram
$$
\xymatrix@=20pt{
\Pi_S(X'_I,v_I)\ar[r]
&\hdots\ar@<3pt>[r]\ar@<-3pt>[r]\ar[r] 
& \bigoplus_{J \subset I, \sharp J=2} \Pi_S(X'_J,v_J)\ar@<2pt>[r]\ar@<-2pt>[r]
& \bigoplus_{i \in I} \Pi_S(X'_i,v_i)
}
$$
Moreover, 
its dual $\cohtp_S(X,-v)$ is isomorphic to the homotopy limit of the diagram
$$
\xymatrix@=20pt{
\bigoplus_{i \in I} \Pi_S(X'_i,-v_i-T_i)\ar@<2pt>[r]\ar@<-2pt>[r]
& \bigoplus_{J \subset I, \sharp J=2} \Pi_S(X'_J,-v_J-T_J)\ar@<3pt>[r]\ar@<-3pt>[r]\ar[r] 
& \hdots\ar[r]
& \Pi_S(X'_I,-v_I-T_I)
}
$$
Here, 
for $J \subset I$, 
$T_J$ (resp. $v_J$) is the tangent bundle of (resp. pullback of $v$ over) $X'_J/S$.
A similar computation is valid in case (2) using \Cref{cor:snccorollary2}.
\end{rem}

\begin{ex}
Combining (ii) with the theorem of embedded resolution of singularities shows that for any smooth separated scheme $X$ 
over a field $k$ of characteristic $0$, 
the spectrum $\htp_k(X,v)$ is rigid with dual $\htp_k^c(X,-v-T_X)$. 
So the above theorem vastly generalizes the case of geometric motives from \cite[Chapter 5, 4.3.7]{FSV}.
\end{ex}

\begin{rem}
When $S$ is of positive dimension, 
one cannot expect all constructible spectra to be rigid (compare with \cite{motiviclandweber,RiouDual}).
Properness or smoothness does not ensure rigidity.
For example, 
let $U \subset S$ be an open subscheme with non-empty complement $Z$.
Assume $i:Z \rightarrow S$ is h-smooth (for example $Z$ and $S$ are smooth over a field $k$).
We claim that $\htp_S(U)=j_!(\un_U)$ is not rigid. 
Assuming the contrary, 
its dual would be isomorphic to $j_*(\un_U)$ according to \Cref{lm:pre-duality}. 
Since $i^*$ is monoidal,
it would follow that $i^*j_!(\un_U)$ is rigid with dual $i^*j_*(\un_U)$.
The first spectrum is trivial, 
whereas purity identifies the second one with $\un_Z \oplus \Th(N_ZS)[1]$ which is nontrivial.
An identical (dual) argument shows that $\htp_S(Z)$ is not rigid.
\end{rem}

Next we show an improvement of \Cref{thm:generalizedhomotopypurity}.

\begin{thm}
\label{thm:generalizedhomotopypurity2}
Consider a separated $S$-scheme $X$ which admits a smooth compactification $\bar p:\bar X \rightarrow S$ 
with tangent bundle $\bar T$ and boundary $\nu:\partial X \rightarrow \bar X$ a normal crossing $S$-scheme.

Then there exists a canonical isomorphism
$$
\htp_S(\bar X/X) \simeq \htp_S(\partial X,-\nu^*\bar T)^{\vee}
$$
Here, 
$\htp_S(\partial X,-\nu^*\bar T)^\vee$ is the dual of the rigid motive $\htp_S(\partial X,-\nu^*\bar T)$
(according to \Cref{thm:strong_duality}). 

Moreover, 
the map $\beta'$ obtained by adjunction from
$$
\beta:\htp_S(\partial X) 
\xrightarrow{\nu_*}  \htp_S(\bar X) 
\rightarrow \htp_S(\bar X/X) 
\simeq \htp_S(\partial X,-\nu^*\bar T)^\vee
$$
fits into the commutative diagram
$$
\xymatrix@=20pt{
\htp_S(\partial X) \otimes \htp_S(\partial X,-\nu^*\bar T)\ar^-{\beta'}[rr]\ar_{Id \otimes \nu_*}[d] && \un_S \\
\htp_S(\partial X) \otimes \htp_S(\bar X,-\bar T)\ar_\simeq^{(*)}[r] & 
\htp_S(\partial X \times_S \bar X,-p_2^*\bar T)\ar^-{\gamma_\nu^!}[r] & 
\htp_S(\partial X)\ar_{q_*}[u]
}
$$
where $\gamma_\nu$ is the graph of the closed immersion $\nu:\partial X \rightarrow \bar X$.
\end{thm}
\begin{proof}
The first assertion is a combination of Theorems 
\ref{thm:generalizedhomotopypurity}, \ref{thm:strong_duality}.
For the second assertion, 
let us note that $\gamma_\nu$ is a section of the smooth separated morphism 
$\Id \times_S \bar p:\partial X \times X \bar X \rightarrow \partial X$.
So it is a regular closed immersion whose normal bundle is isomorphic to the tangent bundle of $\Id \times_S \bar p$,
that is $p_2^*\bar T$. 
This justifies the existence of the Gysin map $\gamma_\nu^!$ in \ref{num:Gysin}.
Secondly, 
the isomorphism ($*$) follows from the computation of $\pi_S(\partial X)$ performed in \Cref{ex:sncexample1}
and the K{\"u}nneth isomorphism.
A routine check using the definitions of the maps show that the diagram commutes.
\end{proof}
 
\begin{num}
\Cref{thm:generalizedhomotopypurity2} gives a new tool for computing stable motivic homotopy types at infinity.
According to \Cref{prop:basic_comput_thp-infty},
$\Pi^\infty_S(X)$ is isomorphic to the homotopy fiber of the map
\[
\beta:\htp_S(\partial X) \rightarrow \htp_S(\bar X/X)
\]
Using the isomorphism in the above theorem,
the commutative diagram tells us that $\beta$ can be computed as the class
$$
(\Id \times_S \nu)^*([\Gamma_\nu]) \in \Pi^{2d,d}(\partial X \times_S \partial X,p_2^* \bar T)
$$
Here $d$ is the dimension of $\bar X$, 
and $[\Gamma_\nu]$ is the fundamental class of the graph of $\nu$ in 
\[
\Pi^{2d,d}(\partial X \times_S \bar X,p_2^* \bar T)
\]
via the canonical isomorphism
\begin{align*}
[
\htp_S(\partial X),\htp_S(\bar X/X)] &\simeq [\htp_S(\partial X),\htp_S(\partial X,-\nu^*\bar T)^\vee] \\
& \simeq [\htp_S(\partial X) \otimes \htp_S(\partial X,-\nu^*\bar T),\un_S]
=
\Pi^{2d,d}(\partial X \times_S \partial X,p_2^* \bar T)
\end{align*}
The latter isomorphism is a consequence of the K\"unneth formula
$$
\htp_S(\partial X) \otimes \htp_S(\partial X,-\nu^*\bar T) \simeq \htp_S(\partial X \times \partial X,-p_2^*\bar T)
$$
which follows from the smooth case and \Cref{ex:sncexample1}.
\end{num}

\subsection{Application to stable motivic homotopy types at infinity}

\begin{num}\label{num:maincomputation}
We are ready for our main computation,
following \Cref{thm:compute_closed_cover} and \Cref{prop:basic_comput_thp-infty}.
Let $X/S$ be separable with compactification $(\bar X,\partial X)$ over $S$.
In addition, 
we assume that $\bar X$ is proper and smooth over $S$ and $\partial X$ is a normal crossing divisor on $\bar X$,
see \Cref{df:normal_crossing}.
We write $\partial X=\cup_{i \in I} \partial_i X$ for the decomposition of $\partial X$ into its irreducible components, 
and we consider $\partial_i X$ as a reduced subscheme of $\bar X$.
There is a canonical closed immersion $\nu_i:\partial_i X \rightarrow \bar X$.
For a subset $J \subset I$, we equip $\partial_J X=\cap_{j \in J} \partial _j X$ with its reduced subscheme structure. 
There is a canonical closed immersion $\nu_K^J:\partial_K X \rightarrow \partial_J X$ for subsets $J \subset K\subset I$.
\end{num}

\begin{thm}
\label{thm:maincomputation}
Consider the above assumptions and notations.
The stable homotopy type at infinity $\Pi_{S}^\infty(X)$ of $X/S$ is canonically isomorphic to the 
homotopy fiber of the map
$$
\colim_{n \in (\Dinj)^{op}} \left(\bigoplus_{J \subset I, \sharp J=n+1} \Pi_{S}(\partial_JX)\right)
\xrightarrow \mu \underset{n \in \Dinj}\lim
\left(\bigoplus_{J \subset I, \sharp J=m+1} \Pi_{S}(\partial_J X,\twist{N_J})\right)
$$
Here, 
the direct images define the face maps in the source
$$
\delta_n^k=\sum_{K=\{i_0<\hdots<i_n\}, J=\{i_0<\hdots<\not{i_k}<\hdots<i_n\}} (\nu_{K}^J)_*
$$
and the Gysin maps define the coface maps in the target
$$
\tilde \delta^m_l=\sum_{K=\{i_0<\hdots<i_m\}, J=\{i_0<\hdots<\not{i_l}<\hdots<i_m\}} (\nu_{K}^J)^!
$$
Moreover, 
$\mu$ is induced by the canonical map in degree zero 
$$
(\nu_j^!\nu_{i*})_{i,j\in I}
\colon \bigoplus_{i \in I} \Pi_{S}(\partial_iX) 
\longrightarrow \bigoplus_{j \in I} \Pi_{S}(\partial_j X,\twist{N_j})
$$
\end{thm}
\begin{proof}
According to \Cref{prop:basic_comput_thp-infty}, the spectrum
$\Pi_{S}^\infty(X)$ is the homotopy fiber of the map $\beta_X\colon\Pi_{S}(\partial X) \rightarrow \Pi_{S}(\bar X,X)$.
\Cref{thm:compute_closed_cover} (see \Cref{ex:sncexample1}) identifies $\mu$'s source with the desired colimit.
An application of Theorems \ref{thm:generalizedhomotopypurity2}, \ref{thm:compute_closed_cover}, 
(see also \Cref{rem:strong_duality})
identifies $\mu$'s target with the desired spectrum on account of the identity 
$$
\twist{\bar T}-\twist{T_{\partial_iX}}=\twist{N_i}
$$
between (virtual) vector bundles over $\partial_i X$.
The computation of the (co)face maps and $\mu$ follows from the definition of the Gysin maps.
\end{proof}

One can suggestively summarize the computation in \Cref{thm:maincomputation} with the diagram
\begin{gather*}
\left(
\vcenter{
\xymatrix@R=10pt{
	\ar@{..}[d] & \\
	\bigoplus_{i_1<i_2} \Pi_{S}(\partial_{i_1i_2}X)\ar@<4pt>[d]\ar@<-4pt>[d] & \\
	\bigoplus_{i \in I} \Pi_{S}(\partial_iX)\ar^-{\mu}[r] 
	& \bigoplus_{j \in I} \Pi_{S}(\partial_jX,\twist{N_j})\ar@<4pt>[d]\ar@<-4pt>[d] \\
	& \bigoplus_{j_1<j_2} \Pi_{S}(\partial_{j_1j_2}X,\twist{N_{j_1j_2}})\ar@{..}[d] \\
	& 
}
}
\right)
\end{gather*}

\section{A motivic analog of Mumford's plumbing game}
\label{section:Mumford}

\begin{num}
\label{num:general_Mumford}
In this section, 
we consider (relative) compactified smooth surfaces whose boundary is a union of (relative) rational curves.
Inspired by the technique of ``plumbing" disk bundles over manifolds, 
this case was first studied by Mumford \cite{mumfordihes} in the context of normal surface singularities, 
e.g., rational double points.
\vspace{0.1in}

We work in the setting of \Cref{num:maincomputation}, 
where $S$ is a base scheme, 
$X/S$ is a smooth family of surfaces with a smooth proper compactification $\bar X/S$,
whose boundary $\partial X=\cup_{i \in I} \partial_i X$ is a normal crossing divisor on $\bar X$ 
(see \Cref{df:normal_crossing}).
We also assume that each branch $\partial_i X$, with its reduced schematic structure,
is a rational curve over $S$.
For each $i\in I$, we fix an $S$-isomorphism $\epsilon_i:\partial_i X \rightarrow \PP^1_S$.
The choice of isomorphism determines $3$ distinct $S$-points on $\partial_i X$, 
say $0,1,\infty$. 
For $i<j$ in $I$, 
we let $\partial_{ij} X$ be the reduction of $\partial_i X \times_S \partial_j X$, 
which is a finite \'etale $S$-scheme.
We consider the immersion $\nu_{ij}^k:\partial_{ij} X \rightarrow \partial_{k} X$,
 for $k=i,j$.
Up to changing the isomorphism $\epsilon_i$ we may assume:
\begin{equation}
\label{eq:Mumford_good_points}
\text{the image of $\epsilon_i \circ \nu_{ij}^k$ does not meet the point at infinity}
\end{equation}
\end{num}

\subsection{Quadratic intersection degrees}

\begin{num}
\label{num:fdl_classes}
Recall from \cite{DJK} that to a regularly embedded closed subscheme $i:Z \rightarrow X$ one can associate 
a fundamental class $\eta_X(Z)$;
it takes value in cohomotopy with support\footnote{This is the map \eqref{eq:fdl_class} for $f=i$,
evaluated at $\un_X$.}
$$
\eta_X(Z) \in \Pi^0_Z(X,\twist{N_ZX}):=[i_*(\Th(-N_ZX)),\un_X](\simeq H_0(Z/X,-N_ZX))
$$
If $X$ and $Z$ are smooth over the base scheme $S$, 
the said class comes via Morel-Voevodsky's homotopy purity theorem
$$
\mathfrak p_{Z/X}:\Pi^0(Z) \simeq \Pi^0_Z(X,\twist{N_ZX}), 1 \mapsto \eta_X(Z)
$$
\end{num}

\begin{df}\label{df:quad_intersect}
Let $S$ be a base scheme, and $X$ a smooth (proper) $S$-scheme.
Let $D \subset X$ be an effective Cartier divisor 
(\emph{i.e.,} in our case a closed subscheme of $X$ of pure codimension $1$) 
and $C/S$ a smooth relative curve with an embedding $i:C \rightarrow X$.
We set $\bar Z=D \times_X C$, $Z=\bar Z_{red}$, and assume the following holds:
\begin{enumerate}
\item The intersection of $D$ and $C$ in $X$ is proper: $\bar Z$ is finite over $S$
\item The projection map $f:Z \rightarrow S$ is 
\'etale\footnote{This is automatic if $S$ is the spectrum of a perfect field.}
\end{enumerate}
The (unoriented) quadratic intersection degree of $D$ and $C$ in $X$ is the class 
$$
(D,C)_{quad}=\tdeg\big(i^*(\eta_X(D)\big)
\in 
\Pi^0(S)
$$
Here, 
$i^*:\Pi^0_D(X,\twist{N_DX}) \rightarrow \Pi^0_Z(C,\twist{N_ZC})$ is the pullback 
map\footnote{The proper intersection assumption gives a canonical isomorphism: $N_ZC \simeq i^*(N_DX)$.}, 
$f_*:\Pi^0(Z) \rightarrow \Pi^0(S)$ is the Gysin map associated with the finite \'etale morphism $f$, 
see \cref{num:Gysin}, 
and the (quadratic) degree map is the composite
$$
\tdeg:=f_* \circ \mathfrak p_{C/Z}^{-1}
$$
\end{df}

In this definition's generality, 
the quadratic intersection degree can be interpreted as a family of quadratic forms. 
Indeed, by a pullback, one can specialize $(D,C)_{quad}\in\Pi^0(S)$ to any point $x \in S$ of the base.
This yields an element of $\Pi^0(x) \simeq \GW(\kappa(x))$, 
the Grothendieck-Witt ring of the field $\kappa(x)$ according to 
\cite{MorLNM}.\footnote{When bad reduction can occur, 
specializing at a point does not preserve the normal crossing divisor.
Nevertheless, for good reduction, 
one can show that our definition of quadratic intersection degree specializes correctly.}
\vspace{0.1in}

For certain base schemes, we can give a more concrete formula for the quadratic intersection degree. 
Let us first note the following consequence of Morel's proof of the Gersten resolution for homotopy sheaves 
(and modules, see \cite{MorLNM}).

\begin{lm}
\label{lm:htp_semi_local}
If $S=\Spec(\mathcal O)$ is an essentially smooth semi-local $k$-scheme, 
then there is a canonical 
isomorphism\footnote{This is generalized to $\ZZ[1/2]$ and other number rings in \cite{2021arXiv210201618B}.}
$$
\GW(\mathcal O)=\Gamma(S,\uKMW) \xrightarrow{\simeq} \Pi^{0,0}(S)
$$
The left-hand-side is the Grothendieck-Witt ring of $\mathcal O$.
\end{lm}

\begin{prop}\label{prop:quadratic-inter}
Let $S=\Spec(\mathcal O)$ be an essentially smooth semi-local scheme over a field $k$.
Under the assumptions in Definition \Cref{num:fdl_classes}, 
let $x$ be a generic point of $Z$. 
We write $Z_x=\Spec(B_x)$ for the connected component of $Z$ containing $x$ and moreover:
\begin{itemize}
\item $m_x=\lg(\mathcal O_{\bar Z,x})$ for the geometric multiplicity of $\bar Z$ at $x$, 
which is the intersection multiplicity of $Z$ and $C$ at $x$
\item $\tau_x$ for the class in the Grothendieck-Witt ring $\GW(\mathcal O)$ of the non-degenerate bilinear form
$$
\varphi_x:B_x \otimes_\mathcal O B_x \rightarrow \mathcal O, x \otimes y \mapsto \mathrm{Tr}_{B_x/\mathcal O}(xy)
$$
Here, 
$\mathrm{Tr}_{B_x/\mathcal O}$ is the trace form of the finite flat $\mathcal O$-algebra 
$B_x$\footnote{Recall the trace form $\varphi$ is non-degenerate if and only if $B_x$ is \'etale over $\mathcal O$.}
\end{itemize}
Then the quadratic intersection degree of $D$ and $C$ in $X$ equals 
$$
(D,C)_{quad}=\sum_{x \in Z^{(0)}} (m_x)_\epsilon\tau_x 
\in 
\Pi^{0,0}(S) \simeq \GW(\mathcal O)
$$
Here, for an integer $n>0$, we set $n_\epsilon=\sum_{i=1}^{n} \langle (-1)^{i+1} \rangle$.
\end{prop}
\begin{proof}
We note that \cite[Theorem 2.2.2]{Feld2} implies the formula
$$
i^*\eta_X(D)=\sum_{x \in Z^{(0)}} (m_x)_\epsilon\eta_X(Z_x)
$$
Indeed, 
by additivity and localisation around $x$, 
one reduces to the case where $Z$ has a single connected component.
The conditions of \emph{loc. cit.} are fulfilled with $e=m_x$, 
according to our assumptions in Definition \Cref{num:fdl_classes}. 
Let $\nu^{(m_x)}:N_ZC \rightarrow i^*N_DX$ be the map given by sending a local parameter to its $m_x$-th power. 
The induced map $\nu^{(m_x)}_*:\Th(N_ZC) \rightarrow \Th(i^*N_DX)$ corresponds to the quadratic form $e_\epsilon$ 
under the isomorphism
$$
\lbrack \Th(N_ZC),\Th(i^*N_DX) \rbrack \simeq \lbrack \un_S,\Th(i^*N_DX-N_ZC)\rbrack 
\simeq \lbrack\un_S,\un_S\rbrack 
\simeq \GW(\mathcal O)
$$
Since $\tdeg$ commutes with multiplication by $(m_x)_\epsilon$, 
it suffices to check that $\tdeg(\eta_C(Z_x))=\tau_x$. 
As recalled in \ref{num:fdl_classes}, 
$\mathfrak p_{Z_x/C}(\eta_C(Z_x))=1$; 
thus we are reduced to showing $f_*(1)=\tau_x$, 
which follows using the computation of transfer maps in \cite{Feld3}.
\end{proof}

\begin{rem}
In particular, 
we have enhanced the classical notion of intersection degree to a quadratic version; 
one recovers the former from the latter by taking ranks of non-degenerate symmetric bilinear forms.
\end{rem}

\begin{num}\label{num:oriented_degree&class}
A merit of Definition \Cref{num:fdl_classes} is that it dispenses with orientation conditions.
To get a more general definition, we will need such conditions.
Suppose $S$ is defined over a field $k$.
We write $\HMW\ZZ_S$ for ring spectrum representing Milnor-Witt cohomology over $S$,
see \cite{2020arXiv200406634B}.
Recall that a virtual bundle $v$ over a scheme $X$ is orientable if
 the line bundle $\det(v)$ is divisible by $2$ in $\Pic(X)$.
 An orientation of $v$ is an isomorphism $\tau:\det(v) \rightarrow L^{\otimes,2}$
 for some line bundle $L/X$.
 Further, if $p:X \rightarrow S$ is an lci morphism
 with virtual tangent bundle $\tau_p$,
 we say that $p$ is orientable if $\tau_p$ is orientable; an orientation of $p$
 is an isomorphism $\tau:\omega_p=\det(\tau_p) \rightarrow L^{\otimes,2}$.

When $p$ is projective of relative dimension $d$, 
one can define the Milnor-Witt degree as the composite
\begin{align*}
\tdeg_{X/S}^\tau:\HMW^{2d,d}(X) 
\xrightarrow{\tau_*} \HMW^{0,0}(X,\twist{L_f})
\xrightarrow{p_*} \HMW^{0,0}(S)
\end{align*}
Here, 
the map $\tau_*$ is the isomorphism induced by  $\tau$,
since $\HMW \ZZ$ is $\SL^c$-orientable.
 
If now $p=i:Z \rightarrow X$ is a regular closed immersion of codimension $n$,
an orientation $\theta$ of $i$ is an orientation of the normal bundle $N_ZX$.
In that case, 
one defines the oriented cycle class $[Z]^\Theta_X$ of $Z$ in $X$ as the image of $\eta_X(Z)$ 
under the composite\footnote{Equivalently, 
when $X$ is smooth it is the image of $1$ under the composite map
$$
\CHt^0(Z) \xrightarrow{\Theta_*} \CHt^0(Z,\det(N_ZX)) \xrightarrow{i_*} \CHt^0(X)
$$}
$$
\Pi^0_Z(X,\dtwist{N_ZX}) 
\rightarrow
\HMWx Z^{0}(X,\twist{N_ZX})
\xrightarrow{\theta_*} \HMWx Z^{0}(X,\twist{n})
\rightarrow \HMW^{2n,n}(X)
$$
Here, 
the first map is induced by the unit of the ring spectrum $\HMW \ZZ_S$ and the second map forgets the support.
\end{num}

\begin{df}
\label{df:orient_quad_intersect}
Suppose $S$ is a $k$-scheme and $X$ a finite type $S$-scheme.
We consider a regular closed immersion $i:Z \rightarrow X$ of codimension $n$ such that $N_ZX$ is orientable, 
and a morphism $f:T \rightarrow X$ such that the projection $p:T \rightarrow S$ is lci projective,
orientable and of relative dimension $n$.
Let $\theta$ (resp. $\tau$) be an orientation of $i$ (resp. $p$).
 
The oriented quadratic intersection degree of $Z$ along $f$ over $S$ is the class 
$$
(Z,T)^{\theta,\tau}_{quad}=
\tdeg_{T/S}(f^*[Z]^\Theta_X)
\in 
\HMW^{0,0}(S).
$$
Though all the previous notions depends on the chosen orientations,
 we will sometimes simply write $[Z]_X^{or}$, $\tdeg^{or}_{X/S}$, $(Z,T)^{or}_{quad}$,
 the chosen orientation being intended.
\end{df}
In the case where $S=\Spec(\mathcal O)$
 is semi-local essentially smooth over $k$, 
 as explained in Lemma \ref{lm:htp_semi_local},
 the class $(Z,T)^{or}_{quad}$ leaves in $\GW(\mathcal O)$.
 
\begin{ex} We consider the notations of the previous definition.
\begin{enumerate}
\item We assume that $f:X \rightarrow S$ is smooth projective and oriented,
 of dimension $2n$,
 $Z$ and $T$ are closed regular subschemes of $X$ of codimension $n$,
 whose normal bundles are oriented.
 In this case, the oriented cycle class defined in \Cref{num:oriented_degree&class}
 are elements $[Z]_X^{or}, [T]_X^{or}$ in the Chow-Witt group $\CHt^n(X)$.
 Then the projection formula gives the following more symmetric formula
 for the oriented quadratic intersection degree:
$$
(Z,T)^{or}_{quad}=\tdeg^{or}_{X/S}([Z]_X^{or}.[T]_X^{or})
$$
using the intersection product of $\CHt^n(X)$ (hint: use the projection formula).
\item Assume that $Z=T$.
 Then according to the excess intersection formula
 \cite[Proposition 3.3.4]{DJK}, we get:
$$
(Z,Z)^{or}_{quad}=\tdeg{Z/S}\big(e^{or}(N_ZX)\big)
$$
where $e^{or}(N_ZX) \in \HMW^{2n,n}(Z)$ is the Euler class of the oriented vector bundle $N_ZX$.
\end{enumerate}
\end{ex}

\subsection{A universal homotopical formula}

\begin{num}
We follow the setting of \ref{num:general_Mumford}.
\Cref{thm:maincomputation} shows $\htp_S(\partial X)$ is the homotopy cofiber of the canonical map
$$
d=\sum_{i<j} \nu^i_{ij*}-\nu^j_{ij*}:
\bigoplus_{i<j} \htp_S(\partial_{ij} X) 
\rightarrow 
\bigoplus_{i \in I} \htp_S(\partial_{i} X)
$$
Moreover, we consider the isomorphism
\begin{equation}
\label{equation:lhs}
\epsilon
=
\sum_{i \in I} \epsilon_{i*}:\bigoplus_{i \in I} \htp_S(\partial_{i} X) 
\rightarrow 
\bigoplus_{i \in I} \un_S \oplus \bigoplus_{i \in I} \un_S(1)[2]
\end{equation}
\end{num}

The following lemma is immediate.
\begin{lm}
\label{lm:Mumford_source}
The composite $\epsilon \circ d$ factors through the inclusion of the direct summand $\bigoplus_{i \in I} \un_S$ 
on the right-hand-side of \eqref{equation:lhs}.
Moreover, 
the factorization coincides with the map
\begin{equation}
\label{eq:top_part_Mumford}
p=\sum_{i<j} p_{ij*}^i-p_{ij*}^j:\bigoplus_{i<j} \htp_S(\partial_{ij} X) \rightarrow \bigoplus_{i \in I} \un_S
\end{equation}
Here, 
$p_{ij}:\partial_{ij} X \rightarrow S$ is the canonical projection (a finite \'etale map),
 and for $k=i,j$, we have written $p_{ij*}^k$ for the projection on the $k$-th factor
 of the right hand-side.
In particular, 
if $\mathcal D_X$ is the homotopy cofiber of $p$,
there is a canonical isomorphism
$$
\htp_S(\partial X) \simeq \mathcal D_X \oplus \bigoplus_{i \in I} \un_S(1)[2]
$$
\end{lm}

One can interpret the decomposition by saying that $\htp_S(\partial X)$ is a sum of the combinatorial 
part $\mathcal D_X$, 
which is a smooth Artin motivic 
spectrum\footnote{By analogy with the case of motives, 
it is the smallest $\infty$-category containing $\Pi_S(V)$ for $V/S$ finite \'etale, 
and stable under suspensions, homotopy (co)fibers.} 
and the ``geometric" part $\oplus_{i \in I} \un_S(1)[2]$.


\begin{num}
Owing to \Cref{thm:maincomputation}, 
we can identify $\htp_S(\bar X/X)$ with the homotopy fiber of the map
$$
d_2=\sum_{i<j} \nu_{ij}^{i!}-\nu_{ij}^{j!}:
\bigoplus_{j \in I} \htp_S(\partial_{j} X,\twist{N_j})
\rightarrow 
\bigoplus_{j<k} \htp_S(\partial_{jk} X,\twist{N_{jk}})
$$
Here, 
$N_j$ (resp. $N_{jk}$) denotes the normal bundle of $\partial_j X$ (resp. $\partial_{jk} X$) in $\bar X$).
The isomorphism $\partial_jX \simeq \PP^1_S$ reduces to considering the homotopy type of $\PP^1_S$ twisted 
by the line bundle on $\PP^1_S$ corresponding to $N_j$ (which is determined by its degree).
\end{num}

\begin{prop}
\label{prop:splittingbylinebundles}
Let $w$ be a virtual vector bundle on $\PP^1_S$ and write $w'=j^*(w)$ for the canonical open immersion 
$j:\AA^1_S \rightarrow \PP^1_S$. 
Suppose $w'$ is constant, 
i.e., 
there exists a virtual vector bundle $w^0$ on $S$ such that $w'=\pi^*(w^0)$, 
where $\pi:\AA^1_S \rightarrow S$ is the 
projection.\footnote{This holds if $S$ is regular by homotopy invariance of $K$-theory of regular schemes. 
Note that $w^0$ is always isomorphic to the restriction of $w$ over the $S$-point zero.}

Then the exact homotopy sequence
$$
\htp_S(\AA^1_S,w') \xrightarrow{j_*} \htp_S(\PP^1_S,w) \rightarrow \htp_S(\PP^1_S/\AA^1_S,w)
$$
is canonically isomorphic to
\begin{equation}\label{eq:twisted_P1}
\Th(w^0) \xrightarrow{s_{0*}} \htp_S(\PP^1_S,w) \xrightarrow{s_\infty^!} \Th(w^\infty)(1)[2]
\end{equation}
Here, 
$w^\infty$ is the restriction of $w$ along the section $s_\infty$ of $\PP^1_S$ at $\infty$, 
and $s_\infty^!$ denotes the associated Gysin morphism.

Assume also that $S=\Spec(\mathcal O)$ is a semi-local essentially smooth scheme over a field $k$.
Then the unit $\eta:\un_S \rightarrow \KMW_*$ of the homotopy module representing Milnor-Witt K-theory 
over $S$ induces an isomorphism
$$
\eta_*:[\htp_S(\PP^1_S,w),\Th(w^0)]=\Pi^0(\PP^1_S,w-\bar \pi^*(w^0)) \xrightarrow{\simeq} \CHt^0(\PP^1_S,L)
$$
The right-hand-side is the Chow-Witt group of $\PP^1_S$ twisted by the line bundle
$$
L=\det(w) \otimes \bar \pi^*(\det(w^0))^{-1}
$$
It follows that \eqref{eq:twisted_P1} splits if and only if $w$ is oriented (see \ref{num:oriented_degree&class}), 
which amounts to assuming $w$ is an even line bundle on
$\PP^1_S$.\footnote{That is, $w$ is isomorphic to $\mathcal O(2n)$.} 
Moreover, 
any choice of orientation of $w$ yields an isomorphism
$$
[\htp_S(\PP^1_S,w),\Th(w^0)] \simeq \GW(\mathcal O)
$$
and therefore a splitting of \eqref{eq:twisted_P1} so that
\[
\htp_S(\PP^1_S,w) \simeq \Th(w^0) \oplus \Th(w^\infty)(1)[2]
\]
\end{prop}
\begin{proof}
The first statement follows by combining our generalized form of the homotopy purity theorem, 
see \Cref{thm:generalizedhomotopypurity2}, 
and the observation that the isomorphisms
\begin{align*}
\Th(w^0)=\htp_S(S,w^0) &\xrightarrow{\sigma_{0*}} \htp_S(\AA^1_S,w')\\
\htp_S(\AA^1_S,w') \simeq \htp_S(\AA^1_S,\pi^*(w^0))
&\xrightarrow{\pi_{*}} \htp_S(S,w^0)=\Th(w^0)
\end{align*}
are mutually inverse to each other (due to our assumption on $w'$).
\vspace{0.1in}

For the second statement we apply the natural transformation $[-,\eta]$ to the exact homotopy sequence 
\eqref{eq:twisted_P1}.
This yields a commutative diagram of exact sequences of abelian groups
(see \cite[2.2.3]{Feld3} for Nisnevich cohomology of the twisted Milnor-Witt sheaf in degree $-1$)
$$
\xymatrix@R=14pt@C=30pt{
0\ar[r] & \Pi^0(\PP^1_S,w-\bar \pi^*w^0)\ar^-{s_0^*}[r]\ar_{\eta_*}[d] & \Pi^0(S,0)\ar^-{s_{\infty*}}[r]\ar^\simeq[d] & 
\Pi^{-1,-1}(S,w^\infty-w^0)\ar^\simeq[d] \\
0\ar[r] & \CHt^1(\PP^1_S,L)\ar^-{s_0^*}[r] & \CHt^0(S)\ar^-{s_{\infty*}}[r] & H^0(S,\KMW_{-1}\{L_S\})
}
$$
The middle isomorphism follows since both groups are isomorphic to $\GW(\mathcal O)$ by \Cref{lm:htp_semi_local}, 
and the right-most isomorphism between copies of $W(\mathcal O)$ follows the same argument 
(Morel's proof of Gersten's conjecture for homotopy sheaves).
Therefore, 
$\eta_*$ is an isomorphism.

Now, 
if $w$ is $\bar \pi$-oriented, 
$L$ being a square and $\CHt^*$ being $\SL^c$-oriented,
one gets for a choice of an $\SL^c$-orientation of $L$ the required canonical isomorphism
$$
\CHt^1(\PP^1_S,L) \simeq \CHt^1(\PP^1_S) \simeq \GW(\mathcal O)
$$
according to \cite[11.7]{FaselPB} for the second isomorphism.
If $L$ is odd, 
then, again by \emph{loc. cit.}, there is an isomorphism $\CHt^1(\PP^1_S,L) \simeq \W(\mathcal O)$,
and $s_{\infty*}$ is isomorphic to the canonical projection $h:\GW(\mathcal O) \rightarrow \W(\mathcal O)$.
Thus the quadratic form $\langle 1 \rangle$ cannot be in the image of 
$s_0^*$.\footnote{If $\mathcal O=k$ is a quadratically closed field, 
the map $h$ is  multiplication by $2$ on $\ZZ$.}
\end{proof}

\begin{num}
\label{num:htp_Mumford}
In the setting of \ref{num:general_Mumford} we assume in addition:
\begin{enumerate}
\item $S=\Spec(\mathcal O)$ is essentially smooth semi-local over a field $k$
\item For each $j \in I$, 
the normal bundle $N_j$ is relatively orientable over $\partial_j X/S$,
and we choose such an isomorphism $o_j:N_j \otimes \omega_{\partial_jX/S} \simeq L^{\otimes 2}$
\end{enumerate}

Owing to \Cref{prop:splittingbylinebundles} we deduce a canonical isomorphism
$$
\epsilon_2^{-1}:
\bigoplus_{j \in I} \htp_S(\partial_{j} X,\twist{N_j}) 
\rightarrow 
\bigoplus_{j \in I} \Th(N_j^0)
\oplus \bigoplus_{j \in I} \Th(N_j^\infty)(1)[2]
$$
\end{num}

We use this to show the following analog of \Cref{lm:Mumford_source}.

\begin{cor}
\label{cor:Mumford_target}
The composite map 
$$
\bigoplus_{j \in I} \Th(N_j^0)
\oplus \bigoplus_{j \in I} \Th(N_j^\infty)(1)[2]
\xrightarrow{\epsilon_2}
\bigoplus_{j \in I} \htp_S(\partial_{j} X,\twist{N_j}) 
\xrightarrow{d_2}
\bigoplus_{j<k} \htp_S(\partial_{jk} X,\twist{N_{jk}})
$$
factors through the inclusion of the second factor of the left-hand-side.
Moreover, 
via this factorization, 
it corresponds to the morphism
$$
p_2=\sum_{j<k} p_{jk}^{j!}-p_{jk}^{k!}:
\bigoplus_{j \in I} \Th(N_j^\infty)(1)[2]
\rightarrow \bigoplus_{j<k} \htp_S(\partial_{jk} X,\twist{N_{jk}})
$$
Here, 
$p_{jk}:\partial_{jk} X \rightarrow S$ is the canonical projection (a finite \'etale map), 
 for $l=j,k$,
 $p_{jk}^{l!}$ is the composition of the inclusion of the $l$-th factor of the left hand-side
 with the obvious Gysin morphim,
and we have identified the twists $\twist{p_{jk}^{-1}N_j^{\infty} \oplus \AA^1}$ and $\twist{N_{jk}}$ 
(as rank two virtual bundles over a semi-local scheme).
In particular, 
the motivic spectrum $\htp_S(\bar X/X)$ is isomorphic to
$$
\mathcal D'_X \oplus \bigoplus_{j \in I} \Th(N_j^0)
$$
Here, 
$\mathcal D_X'$ is the homotopy fiber of $p_2$.
\end{cor}


We can now refine \Cref{thm:maincomputation} with the help of \Cref{lm:Mumford_source} and 
\Cref{cor:Mumford_target}.
 
\begin{thm}
\label{thm:smhatoomatrix}
In the setting of \Cref{num:general_Mumford} and \Cref{num:htp_Mumford}, 
the stable motivic homotopy type at infinity $\htp^\infty_S(X)$ is the homotopy fiber of the map 
$$
\beta:
\mathcal D_X \oplus \bigoplus_{i \in I} \un_S(1)[2] 
\rightarrow 
\mathcal D'_X \oplus \bigoplus_{j \in I} \Th(N_j^0)
$$
given by the matrix 
$$
\begin{pmatrix}
a & b' \\
b & \mu
\end{pmatrix}
$$
The map $\mu:\bigoplus_{i \in I} \un_S(1)[2] \rightarrow \bigoplus_{j \in I} \Th(N_j^0)$
is computed by the ``quadratic Mumford matrix"
$$
(\mu_{ij})_{1\leq i,j,\leq n}
$$
defined by the composite
\begin{equation}\label{eq:Mumford_coefficients}
\mu_{ij}:\un_S(1)[2]
\stackrel{\iota_i} \hookrightarrow \htp_S(\partial_{i} X) 
\xrightarrow{\nu_{i*}} \htp_S(\bar X) 
\xrightarrow{\nu_{j}^!} \htp_S(\partial_{j} X,\twist{N_j})
\stackrel{\pi_j} \twoheadrightarrow
\Th(N_j^0)   
\end{equation}
Here, 
$\iota_i$ (resp. $\pi_j$) is the (canonical) split monomorphism (resp. epimorphism) in the isomorphism 
$\epsilon_1$ (resp. $\epsilon_2$) of \Cref{lm:Mumford_source} (resp. \Cref{num:htp_Mumford}).
Moreover, 
if one chooses a trivialization of $N_j^0$ over $S$,
$\mu_{ij}$ becomes an element of $\GW(\mathcal O)$ such that 
$$
\mu_{ij}=(\partial_jX,\partial_iX)^{or}_{quad}
$$
\end{thm}

Here, 
we use the natural orientation of the tangent space of $\partial_iX\stackrel{\epsilon_i}\simeq \PP^1_k$ 
(that is, ${\mathcal O}_{\PP^1_k}(2)$) and the chosen orientation of $N_j$ to define the oriented 
quadratic intersection degree on the right hand-side (see Definition \ref{df:orient_quad_intersect}).
Our identification of $\mu_{ij}$ follows from the definitions:
$\nu_j^!$ is essentially determined by the fundamental class $\eta_{\bar X}(\partial_j X)$, 
and the composition with $\nu_{i*}$ is to be understood as the pullback by $\nu_i$ in the homotopy category.

\subsection{Triangulated and abelian mixed motives}

\begin{num}
The motive of a Thom space depends only on the rank of the vector bundle, 
i.e., 
for $\Lambda$-linear motives we do not need to assume normal bundles are relatively oriented.
Thus the motive at infinity $M^\infty(X)$ 
-- or Wildeshaus' boundary motive --
is the homotopy fiber of the canonical map
$$
M(\partial X) \rightarrow M(\partial X)^\vee(2)[4]
$$
associated with the fundamental class of the diagonal of $\partial X/S$ 
(see \Cref{thm:generalizedhomotopypurity2}).
Let $\mathcal D_X^M$ be the smooth Artin motive over $S$ defined as the homotopy cofiber of 
\eqref{eq:top_part_Mumford} seen as a rigid object in $\DM(S,\Lambda)$.
\Cref{thm:maincomputation} implies the next result.
\end{num}

\begin{thm}
\label{thm:smhatoomatrix2}
In the setting of \Cref{num:general_Mumford}, 
the motive at infinity $M^\infty_S(X)$ of the relative surface $X/S$ is the homotopy fiber of the map 
$$
\beta:
\mathcal D^M_X \oplus \bigoplus_{i \in I} \un_S(1)[2] 
\xrightarrow{\begin{pmatrix}
a & b^\vee(2)[4] \\
b & \mu
\end{pmatrix}}
(\mathcal D^M_X)^\vee(2)[4] \oplus \bigoplus_{j \in I} \un_S(1)[2]
$$
Moreover, 
$\mu$ is given by the matrix $(\mu_{ij})_{i,j\in I}$ where
$$
\mu_{ij}=\deg(\nu_i^*(\eta_{\bar X}(\partial_j X)) 
\in 
H^{00}_M(S,\Lambda)
$$
Here, 
the quadratic intersection degree is an integer, $\mu_{ij}\in\ZZ$, in the following cases:
\begin{itemize}
\item $S$ is regular connected and $\Lambda=\QQ$ (use \cite{CD3})
\item $S$ is smooth connected over a field or a Dedekind ring and $\Lambda=\ZZ$ (use \cite{SpiMod})
\end{itemize}
\end{thm}

Note that $M^\infty_S(X)$ is a triangulated mixed and smooth Artin-Tate motive.
Next we relate the above to the motivic $t$-structure, 
which exists for rational coefficients $\Lambda=\QQ$ when:
\begin{itemize}
\item $S=\Spec(K)$, $K$ is a finite field, or a global field (see \cite{LevineAT})
\item $S=\Spec(O_K)$, $\mathcal O_K$ is a number ring (see \cite{ScholbachAT})
\end{itemize}

The homology motive at infinity of $X/S$ in the abelian category of Artin-Tate mixed motives $\MMAT(S,\QQ)$ 
is defined by 
$$
\HM_i^\infty(X):=\HM_i^\mu(M_S^\infty(X))
$$
where $\HM_i^\mu$ is the $i$-th homology functor for the motivic t-structure
(with standard homological conventions).
\begin{cor}
\label{cor:Mumford_ATmotives}
In the notation of \Cref{thm:smhatoomatrix2}, 
the homology motive at infinity $\HM_i^\infty(X)$ vanishes for $i\not\in [0,3]$ and there is an exact sequence 
in the abelian category $\MMAT(S,\QQ)$ 
\begin{align*}
0 \rightarrow &\HM_3^\infty(X) \rightarrow \bigoplus_{i \in I} \un_S(2)
\xrightarrow{\sum_{i<j} p_{ij}^{i!}-p_{ij}^{j!}} \bigoplus_{i<j} M_S(\partial_{ij} X)(2)  \\
& \rightarrow \HM_2^\infty(X) \rightarrow \bigoplus_{i \in I} \un_S(1) 
\xrightarrow{\ \mu\ } \bigoplus_{j \in I} \un_S(1) \\
& \rightarrow \HM_1^\infty(X) \rightarrow \bigoplus_{i<j} M_S(\partial_{ij} X)
\xrightarrow{\sum_{i<j} p^i_{ij*}-p^j_{ij*}} \bigoplus_{i \in I} \un_S 
\rightarrow \HM_0^\infty(X) \rightarrow 0
\end{align*}
where $\mu$ is the Mumford matrix,
 and $M_S(\partial_{ij}X)=\HM^\mu_0(M_S(\partial_{ij}X))$ is seen
 as an abelian Artin-Tate motive.
\end{cor}
Note in particular that $\HM_0^\infty(X)$ and $\HM_3^\infty(X)$ are pure of respective weights $0$ and $-4$,
while $\HM_1^\infty(X)$ and $\HM_2^\infty(X)$ are in general mixed of weights $\{0,-2\}$ and $\{-2,-4\}$, 
respectively.

\subsection{Three general lines in projective plane}
\label{subsection:cycle-0} Over a field $k$, let $L_1,L_2,L_3\cong \mathbb{P}^1_k$ be three general lines in $\mathbb{P}^2_k$. Let $\sigma:\bar{X}\to \mathbb{P}^2_k$ be the surface obtained by blowing-up one $k$-rational point on each of the lines $L_i$ distinct from the $k$-rational points 
$p_{12}=L_1\cap L_2$, $p_{13}=L_1\cap L_3$, $p_{2,3}=L_2\cap L_3$. Then $\bar{X}$ is a del Pezzo surface of degree $6$ and the proper transform $\partial X$ of $L_1\cup L_2\cup L_3$ on $\bar{X}$ a strict normal crossing divisor; it is the support of an anti-canonical divisor on $\bar{X}$ and thus an ample divisor. We now compute the homology motive at infinity of the affine surface 
$X=\bar{X}- \partial X$. In the exact sequence of Corollary \ref{cor:Mumford_ATmotives} the homomorphisms 
$\bigoplus_{i<j}M_k(\partial_{ij}(X))\cong \bigoplus_{i=1}^{3}  \un_k\rightarrow \bigoplus_{i\in I} \un_k \cong \bigoplus_{i=1}^{3} \un_k$ 
and 
$ \bigoplus_{i\in I} \un_k(2) \cong \bigoplus_{i=1}^{3} \un_k(2) \rightarrow \bigoplus_{i<j}M_k(\partial_{ij}(X))(2)\cong \bigoplus_{i=1}^{3} \un_k(2)$ are given respectively by the matrix $$N=\left(\begin{array}{ccc} 1 & 1 & 0 \\ -1 & 0 & 1\\ 0 & -1 & -1 \end{array}\right)$$ and its transpose. Moreover, since $L_i$ has trivial normal bundle in $\bar{X}$ for $i=1,2,3$, 
the quadratic Mumford intersection matrix equals 
$$
\mu=
\left(\begin{array}{ccc} 0 & \langle 1\rangle & \langle 1\rangle \\ 
\langle 1\rangle & 0 & \langle 1\rangle \\ 
\langle 1\rangle & \langle 1\rangle & 0 
\end{array}\right)
$$ 
The matrix $N\in \mathcal{M}_{3,3}(\mathbb{Z})$ is equivalent to
 the diagonal matrix $\mathrm{diag}(1,1,0)$
 and the matrix $\mu$ in $\mathcal{M}_{3,3}(\mathrm{GW}(k))$ is equivalent to the diagonal matrix $\mathrm{diag}(1,1,2)$. In summary, the above yields 
$$\underline{H}_i^{\infty}(X)= 
\left\{ \begin{array}{ll} 
\un_k &  i=0 \\
(\un_k/2)(1) \oplus \un_k &  i=1 \\
\un_k(2) &  i=2 \\
\un_k(2) & i=3
\end{array} 
\right. 
$$

\subsection{Danielewski hypersurfaces}
\label{subsection:examples}

For a field $k$ and $n\geq 1$, 
the Danielewski hypersurface $D_{n}$ is the smooth affine surface $D_{n}$ in $\mathbf{A}_{k}^{3}$ cut out by 
the equation $x^{n}z=y(y-1)$. 
Owing to \cite{Danielewskisurfaces}, 
$D_{n}$ becomes a Zariski locally trivial $\mathbf{G}_{a}$-bundle over the affine line with two origins 
$\breve{\mathbf{A}}_{k}^{1}$ 
(using the factorization of the surjective projection $\pi_{n}=\mathrm{pr}_{x}:D_{n}\to\mathbf{A}_{k}^{1}$).  
Thus $D_{n}$ is $\mathbf{A}^{1}$-equivalent to $\breve{\mathbf{A}}_{k}^{1}$ and $\mathbf{P}_{k}^{1}$. 
The threefolds $D_n\times \mathbb{A}^1_k$ are isomorphic, 
but the surfaces $D_n$ are pairwise non-isomorphic.
Over $\mathbb{C}$, 
Danielewski \cite{Danielewskisurfaces}, Fieseler \cite{fieseler} established this by showing 
the underlying complex analytic manifolds have non-isomorphic first singular homology groups at infinity. 
Our methods provide a base field independent argument which allows to distinguish between the $D_n$'s via 
homology motives at infinity. 
\vspace{0.1in}

We begin by constructing explicit smooth projective completions $\bar{D}_{n}$ of the surfaces $D_{n}$, 
whose boundaries are strict normal crossing divisors. 
The morphism $\varphi_{n}=\mathrm{pr}_{x,y}:D_{n}\to\mathbf{A}_{k}^{2}$ expresses $D_{n}$ as the 
affine modification of $\mathbf{A}_{k}^{2}$ with center at the closed subscheme $Z_{n}$ with ideal 
$(x^{n},y(y-1))$ and divisor $D_{n}=\mathrm{div}(x^{n})$, 
cf. \cite{zbMATH07149737}.
Furthermore, 
$\varphi_{n}$ decomposes into a sequence of affine modifications
\begin{equation}
\label{equation:affinemodifications}
\varphi_{n}
=
\varphi_{1}\circ\psi_{2}\cdots\circ\psi_{n} 
\colon 
D_{n}\to D_{n-1}\to\cdots D_{2}\to D_{1}
\to
{\mathbf{A}}_{k}^{2}
\end{equation}
given by $\psi_{\ell}:D_{\ell}\to D_{\ell-1}$; $(x,y,z)\mapsto(x,y,xz)$, 
with center at the closed subscheme $Y_{\ell-1}=(x,z)$ and divisor $H_{\ell}=\mathrm{div}(x)$.
That is, 
$\varphi_{1}:D_{1}\to\mathbf{A}_{k}^{2}$ is the birational morphism obtained by blowing-up the points $(0,0)$, 
$(0,1)$ in $\mathbf{A}_{k}^{2}$ and removing the proper transform of $\{0\}\times\mathbf{A}_{k}^{1}$, 
and $\psi_{\ell}:D_{\ell}\to D_{\ell-1}$ is the birational morphism obtained by blowing-up the points $(0,0,0)$, 
$(0,0,1)$ in $\pi_{\ell}^{-1}(0)$ and removing the proper transform of $\pi_{\ell-1}^{-1}(0)$.

Now consider the open embedding $\mathbf{A}_{k}^{2}\hookrightarrow\mathbf{P}_{k}^{1}\times\mathbf{P}_{k}^{1}$;
$(x,y)\mapsto([x:1],[y:1])$. 
Then $C_{\infty}=\mathbf{P}_{k}^{1}\times[1:0]$ and $F_{\infty}=[1:0]\times\mathbf{P}_{k}^{1}$ are irreducible
components of $\mathbf{P}_{k}^{1}\times\mathbf{P}_{k}^{1}$ and we set $F_{0}=[0:1]\times\mathbf{P}_{k}^{1}$. 
Let $\bar{\varphi}_{1}:\bar{D}_{1}\to\mathbf{P}_{k}^{1}\times\mathbf{P}_{k}^{1}$ be the blow-up of the points 
$([0:1],[0:1])$, $([0:1],[1:1])$ in $F_{0}$, with respective exceptional divisors $E_{1,0}$, $E_{1,1}$.
From now on the proper transform of $F_{0}$ in $\bar{D}_{1}$ is also denoted by $F_{0}$.
With these definitions, there is a commutative diagram 
$$
\xymatrix@=24pt{
D_{1} \ar_{\varphi_{1}}[d]\ar[r] & \bar{D}_{1} \ar^{\bar{\varphi}_{1}}[d] \\
\mathbf{A}^{2} \ar[r] & \mathbf{P}^{1}\times\mathbf{P}^{1}
}
$$
Here, 
$D_{1}\hookrightarrow\bar{D}_{1}$ is the open immersion given by the complement of the support of the 
strict normal crossing divisor $\partial D_{1}=C_{\infty}\cup F_{\infty}\cup F_{0}$.
The closures in $\bar{D}_{1}$ of the two irreducible components $\{x=y=0\}$
and $\{x=y-1=0\}$ of $\pi_{1}^{-1}(0)$ equal the exceptional divisors $E_{1,0}$ and $E_{1,1}$, 
respectively. 
We calculate the self-intersection numbers $C_{\infty}^{2}=F_{\infty}^{2}=0$, $F_{0}^{2}=-2$ in $\bar{D}_1$; 
that is, 
the usual degrees of the respective normal line bundles of these curves in $\bar{D}_1$, 
see e.g., 
\cite[Chapter 5.6]{zbMATH00790015}, \cite[Chapter IV]{zbMATH06176082}.

To construct $\bar{D}_{n}$, 
$n\geq2$, 
we start with $\bar{D}_{1}$ and proceed inductively by performing the same sequence of blow-ups as 
for the affine modifications $\psi_{\ell}:D_{l}\to D_{\ell-1}$ in \eqref{equation:affinemodifications}.
This yields birational morphisms $\bar{\psi}_{\ell}:\bar{D}_{\ell}\to\bar{D}_{\ell-1}$
consisting of the blow-up of one point on $E_{\ell,0}-E_{\ell-1,0}$ and another point on 
$E_{\ell,1}-E_{\ell-1,1}$ with respective exceptional divisors $E_{\ell+1,0}$ and $E_{\ell+1,1}$
(by convention $E_{0,0}=E_{0,1}=F_{0})$. 
Moreover, 
$D_{\ell}$ embeds into $\bar{D}_{\ell}$ as the complement of the support of the strict normal crossing divisor 
$\partial D_{\ell}=C_{\infty}\cup F_{\infty}\cup F_{0}\cup\bigcup_{i=1}^{\ell-1}(E_{i,0}\cup E_{i,1})$ in such 
a way that the closures of the two irreducible components $\{x=y=0\}$ and $\{x=y-1=0\}$ of $\pi_{\ell}^{-1}(0)$ 
coincide with the divisors $E_{\ell+1,0}$ and $E_{\ell+1,1}$, respectively.
By construction, there is a commutative diagram 
$$
\xymatrix@=22pt{
\bar{D}_{\ell} \ar^-{\bar{\psi_{\ell}}}[r] & \bar{D}_{\ell-1} \ar[r] & \cdots \ar[r] & \bar{D}_{2} 
\ar^-{\bar{\psi_{2}}}[r]  & \bar{D}_{1} \ar^-{\bar{\varphi}_{1}}[r] & \mathbf{P}^{1}\times\mathbf{P}^{1}\\
D_{\ell} \ar^-{\psi_{\ell}}[r] \ar[u] & D_{\ell-1} \ar[r] \ar[u] & \cdots \ar[r] & 
D_{2} \ar^-{\psi_{2}}[r] \ar[u] & D_{1} \ar^-{\varphi_{1}}[r] \ar[u] & \mathbf{A}^{2} \ar[u] }
$$
For every $n\geq2$, 
we may visualize the boundary divisor $\partial D_{n}$ as a fork of $\mathbf{P}^{1}$'s 
\[
\xymatrix{ 
& & & (E_{1,0},-2) \ar@{-}[r] & \cdots \ar@{-}[r]  & (E_{n-1,0},-2) \\
(F_{\infty},0) \ar@{-}[r] & (C_{\infty},0) \ar@{-}[r] & (F_{0},-2) \ar@{-}[ur] \ar@{-}[dr] \\
& & & (E_{1,1},-2) \ar@{-}[r] & \cdots \ar@{-}[r] & (E_{n-1,1},-2) \\}
\]
with the indicated self-intersection numbers for each irreducible component. 
We may order the irreducible components of $\partial D_n$ as follows
$$
F_\infty < C_\infty < F_0  < E_{1,0} < \ldots < E_{n-1,0} < E_{1,1} < \ldots < E_{n-1,1}
$$
The Euler class of $\mathcal{O}_{\mathbf{P}^{1}}(2)$ equals the hyperbolic plane 
$\mathbb{H}=\langle1,-1\rangle\in \GW(k)$.
Thus the Euler class of $\mathcal{O}_{\mathbf{P}^{1}}(-2)$ is $-\mathbb{H}$. 
Moreover, by Proposition \ref{prop:quadratic-inter}, 
the transversal intersection of two copies of $\mathbf{P}^{1}$ in a unique $k$-rational point 
yields the class $\langle1\rangle\in\GW(k)$.
The quadratic Mumford intersection matrix thus takes the form (with zero entries mostly left out of the notation) 
\[
\mu_{n}=\left(\begin{array}{ccccccccc}
0 & \langle1\rangle\\
\langle1\rangle & 0 & \langle1\rangle\\
 & \langle1\rangle & -\mathbb{H} & \langle1\rangle &  0 &  & \langle1\rangle\\
 &  & \langle1\rangle & -\mathbb{H} & \langle1\rangle \\
 &  &  0 & \langle1\rangle & \ddots & \langle1\rangle\\
 &  &  &  & \langle1\rangle & -\mathbb{H} & 0\\
 &  & \langle1\rangle &  &  & 0 & -\mathbb{H} & \langle1\rangle\\
 &  &  &  &  &  & \langle1\rangle & \ddots & \langle1\rangle \\
 &  &  &  &  &  &  & \langle1\rangle & -\mathbb{H}
\end{array}\right)
\]
 
\begin{prop} 
\label{prop:hmDsurfaces}
Over a field $k$ and $n\geq 1$, 
the homology motives at infinity of the Danielewski surfaces $D_n$ are given by 
$$\underline{H}_i^{\infty}(D_n)= 
\left\{ \begin{array}{ll} 
\un_k &  i=0 \\
(\un_k/2n)(1) &  i=1 \\
0 &  i=2 \\
\un_k(2) & i=3
\end{array} \right. 
$$
\end{prop}
\begin{proof}
The boundary $\partial D_n$ has $2n+1$ irreducible components which intersect in $2n$ $k$-rational points. As in Example \ref{subsection:cycle-0}, one infers from the description of $\partial D_n$ that the homomorphism $$\bigoplus_{i<j}M_k(\partial_{ij}(D_n))\cong \bigoplus_{i=1}^{2n}  \un_k\rightarrow \bigoplus_{i\in I} \un_k \cong \bigoplus_{i=1}^{2n+1} \un_k$$ 
is given by an element $N$ of $\mathcal{M}_{2n,2n+1}(\mathbb{Z})$ equivalent to $\mathrm{I}_{2n}$. This implies that $\underline{H}_0^\infty(D_n)=\un_k$ and $\underline{H}_1^\infty(D_n)=\mathrm{Coker}(\mu_n)$. Since the homomorphism $$ \bigoplus_{i\in I} \un_k(2) \cong \bigoplus_{i=1}^{2n+1} \un_k(2) \rightarrow \bigoplus_{i<j}M_k(\partial_{ij}(D_n))(2)\cong \bigoplus_{i=1}^{2n} \un_k(2)$$ is given by the transpose of $N$, it follows that $\underline{H}_3^\infty(D_n)=\un_k(2)$ and $\underline{H}_2^\infty(D_n)=\mathrm{Ker}(\mu_n)$. 
Finally, one can check by elementary transformations on rows and columns that as an element of $\mathcal{M}_{2n+1,2n+1}(\mathrm{GW}(k))$, the quadratic Mumford matrix $\mu_n$ is equivalent to the diagonal matrix $\mathrm{diag}(1,\ldots ,1, n\mathbb{H})$. This implies in turn that  $\underline{H}_2^\infty(D_n)=0$ and $\underline{H}_3^\infty(D_n)=(\un_k/2n)(1)$.
\end{proof}
%
%

\bibliographystyle{amsalpha}
\bibliography{htpinfty}

\end{document}